\documentclass{amsart}

\usepackage[margin=1in]{geometry}

\usepackage{graphicx}
\usepackage{amsmath,amssymb,amsfonts}
\usepackage{mathrsfs}
\usepackage{amsthm}
\usepackage[colorlinks,allcolors=blue]{hyperref}
\usepackage{enumerate} 
\usepackage{bbm} 
\usepackage{bm} 
\usepackage{tikz} 
\usetikzlibrary{cd} 

\usepackage{thmtools, thm-restate} 

\newtheorem{theorem}{Theorem}[section]

\newtheorem{lemma}[theorem]{Lemma}
\newtheorem{proposition}[theorem]{Proposition}
\newtheorem{corollary}[theorem]{Corollary}
\theoremstyle{definition}
\newtheorem{definition}[theorem]{Definition}
\newtheorem{remark}[theorem]{Remark}
\newtheorem{example}[theorem]{Example}
\numberwithin{equation}{section}

\newcommand{\N}{\mathbb{N}}
\newcommand{\Z}{\mathbb{Z}}

\newcommand{\C}{\mathbb{C}}

\newcommand{\F}{\mathbb{F}}

\renewcommand{\P}{\mathbb{P}}

\newcommand{\bfP}{\bm{P}}

\newcommand{\mX}{\mathcal{X}}

\newcommand{\mU}{\mathcal{U}}

\newcommand{\M}{\mathcal{M}}

\newcommand{\es}{\emptyset}

\renewcommand{\tilde}{\widetilde}
\renewcommand{\hat}{\widehat}

\newcommand{\floor}[1]{\left\lfloor #1 \right\rfloor}

\DeclareMathOperator*{\E}{\text{\Large $\mathbb{E}$}}

\newcommand{\ind}{\mathbbm{1}}

\newcommand{\innprod}[2]{\left\langle #1, #2 \right\rangle}
\newcommand{\norm}[2]{\left\| #2 \right\|_{#1}}

\title[Furstenberg--S\'{a}rk\"{o}zy theorem and partition regularity over finite fields]{Furstenberg--S\'{a}rk\"{o}zy theorem and partition regularity of polynomial equations over finite fields}

\date{\today}

\author{Ethan Ackelsberg}
\address{\'{E}cole Polytechnique F\'{e}d\'{e}rale de Lausanne (EPFL), 1015 Lausanne, Switzerland}
\email{ethan.ackelsberg@epfl.ch}

\author{Vitaly Bergelson}
\address{Ohio State University, Columbus, OH 43210 USA}
\email{vitaly@math.ohio-state.edu}



\begin{document}

\maketitle

\begin{abstract}
	We prove new combinatorial results about polynomial configurations in large subsets of finite fields.
	An analogue of the Furstenberg--S\'{a}rk\"{o}zy theorem was established over finite fields in \cite{blm}, where the authors show that for any polynomial $P(x) \in \Z[x]$ with $P(0) = 0$, if $A \subseteq \F_q$ is a subset of a $q$-element finite field and $A$ does not contains distinct $a, b$ such that $b - a = P(x)$ for some $x \in \F_q$, then $|A| = o(q)$.
	In fields of sufficiently large characterstic, the bound $o(q)$ can be improved to $O(q^{1/2})$ by the Weil bound.
	We match this bound in the low characteristic setting and give a complete algebraic characterization of the class of polynomials $P(x) \in \Z[x]$ for which the Furstenberg--S\'{a}rk\"{o}zy theorem holds over finite fields of characteristic $p$ for each prime $p$.
	
	Our next main result deals with an enhancement of the Furstenberg--S\'{a}rk\"{o}zy theorem over finite fields.
	Another consequence of the Weil bound is that if $P(x) \in \Z[x]$ and $A, B \subseteq \F_q$ and there do not exist elements $a \in A$ and $b \in B$ with $b - a = P(x)$ for some $x \in \F_q$, then $|A| |B| = O(q)$, provided that the characteristic of $\F_q$ is sufficiently large depending on $P$.
	We provide a complete description of the family of polynomials for which this asymmetric enhancement of the Furstenberg--S\'{a}rk\"{o}zy theorem holds over fields of characteristic $p$ with $p$ a fixed prime, achieving the same quantitative bounds that are available in the high characteristic setting.
	The class of polynomials we deal with for this problem is intimately connected with the equidistributional behavior of polynomial sequences in characteristic $p$ studied in \cite{bl}.
	
	The exponential sum estimates that we produce in dealing with the above problems also allow us to establish partition regularity of families of polynomial equations over finite fields.
	As an example, we are able to prove: if $P(x) \in \Z[x]$ with $P(0) = 0$, then for any $r \in \N$, there exists $N = N(P,r)$ and $c = c(P,r) > 0$ such that if $q > N$ (with no restriction on the characteristic) and $\F_q = \bigcup_{i=1}^r{C_i}$, then there are at least $cq^2$ monochromatic solutions to the equation $P(x) + P(y) = P(z)$.
\end{abstract}


\section{Introduction}


The goal of this paper is to develop a systematic approach to combinatorial problems dealing with polynomial configurations over finite fields.
An impetus for studying polynomial configurations comes from the Furstenberg--S\'{a}rk\"{o}zy theorem \cite{furstenberg, sarkozy}, which states that any set of integers with positive density contains a square difference (or, more generally, a difference equal to the value of an integer polynomial with zero constant term).
The Furstenberg--S\'{a}rk\"{o}zy has a meaningful variant over finite fields \cite{blm}, which we seek to refine and improve in a variety of ways.
The main regime of interest for us is when the polynomials involved are of high degree relative to the characteristic of the finite field, which introduces a number of complications that are not present for polynomials of low degree and which has not been as thoroughly treated as the low degree case.
Our ideas draw inspiration from recurrence phenomena in ergodic theory and utilize estimates on exponential sums in finite fields.
We combine the classical Weil bound on exponential sums in finite fields with new algebraic tools for handling polynomials of high degree to produce a dichotomy in the behavior of exponential sums involving polynomials of \emph{arbitrary degree} (see Theorem \ref{thm: exponential sum} below).
This has several combinatorial implications, some of which we highlight here:

\begin{itemize}
	\item	We give a characterization of the family of polynomials satisfying the Furstenberg--S\'{a}rk\"{o}zy theorem over finite fields (which we term \emph{finite field intersective polynomials}).
		Moreover, for finite field intersective polynomials $P(x) \in \Z[x]$, we establish a sharp power-saving bound on the maximal size of a subset $A$ of a finite field $\F_q$ such that $A$ does not contain any differences equal to a value of $P$, i.e., there are no pairs of distinct elements $a, b \in A$ with $b - a = P(x)$ for some $x \in \F_q$.
	\item	We provide a characterization and prove a sharp power-saving bound for the family of polynomials $P(x) \in \Z[x]$ satisfying an asymmetric version of the Furstenberg--S\'{a}rk\"{o}zy theorem where the elements $a$ and $b$ satisfying $b - a = P(x)$ are taken from potentially distinct sets $A$ and $B$.
	\item	We prove new Ramsey-theoretic results about polynomial equations over finite fields.
            For example, we show that the polynomial equation $P(x) + P(y) = P(z)$ is partition regular over finite fields for polynomials $P(x) \in \Z[x]$ with $P(0) = 0$.
\end{itemize}

After introducing some notation, we turn to a more in-depth discussion of our results below.

\subsection{Notation}

In this paper, we make use of the following asymptotic notation for functions on $\N$.
We write $f(n) \ll g(n)$ or $f(n) = O(g(n))$ if there exists a constant $C > 0$ such that $|f(n)| \le C |g(n)|$ for all sufficiently large $n \in \N$.
We use subscripts in expressions such as $f(n) \ll_P g(n)$ to indicate the parameters $P$ on which the implicit constant $C$ depends.
The ``little o'' notation $f(n) = o(g(n))$ means that $f$ grows slower than $g$ in the sense that $\lim_{n \to \infty} \frac{|f(n)|}{|g(n)|} = 0$.

Given a finite set $S$ and a function $f : S \to \C$, we write
\begin{equation*}
	\E_{s \in S} f(s) = \frac{1}{|S|} \sum_{s \in S} f(s)
\end{equation*}
to denote the average of $f$ over $S$, and
\begin{equation*}
	\norm{L^2(S)}{f} = \left( \E_{s \in S} |f(s)|^2 \right)^{1/2}
\end{equation*}
for the $L^2$ norm of $f$ with respect to the normalized counting measure on $S$.


\subsection{The Furstenberg--S\'{a}rk\"{o}zy theorem over finite fields}

The starting point for our discussion is the following version of the Furstenberg--S\'{a}rk\"{o}zy theorem \cite{furstenberg, sarkozy} in the context of finite fields.\footnote{The full statement of \cite[Theorem 5.16]{blm} is a version of the polynomial Szemer\'{e}di theorem over finite fields.
To be precise, given any finite family of polynomials $P_1(x), \dots, P_m(x) \in \Z[x]$ with $P_i(0) = 0$, if $A \subseteq \F_q$ does not contain $\{x, x + P_1(y), \dots, x + P_m(y)\}$ for some $y \ne 0$, then $|A| = o(q)$.
We do not pursue refinements of the full theorem in this paper, so our focus will be on the $m=1$ case.
For quantitative improvements for general $m \in \N$ under some additional conditions on $P_1, \dots, P_m$, see \cite{ab}.}

\begin{theorem}[cf. {\cite[Theorem 5.16]{blm}}] \label{thm: field Sarkozy}
	Let $P(x) \in \Z[x]$ be a polynomial with $P(0) = 0$.
	For any prime power $q$, if $A \subseteq \F_q$ does not contain distinct $a, b$ with $b - a = P(x)$ for some $x \in \F_q$, then $|A| = o(q)$.
\end{theorem}

If one adds the additional assumption that the characteristic of $\F_q$ is greater than the degree of $P$, then one can establish quantitative bounds on the size of the set $A$ in the conclusion of Theorem \ref{thm: field Sarkozy} relatively easily using classical estimates on the size of exponential sums in finite fields.
In particular, by the Weil bound (in the form given in \cite[Theorem 3.2]{kowalski}), if $A \subseteq \F_q$ does not contain distinct $a, b$ with $b - a = P(x)$ for some $x \in \F_q$ and the characteristic of $\F_q$ is larger than the degree of $P$, then
\begin{equation} \label{eq: square root bound}
    |A| \ll_d q^{1/2}.
\end{equation}
However, the case when the degree of $P$ is larger than the characteristic of $\F_q$ is more delicate and requires extra care.
Recent work of Li and Sauermann \cite{ls} nevertheless establishes a power saving bound for Theorem \ref{thm: field Sarkozy} in the low characteristic setting\footnote{Li and Sauermann in fact prove a stronger result that applies to subsets $A \subseteq \F_q[t]$ of polynomials of degree less than $N$. The finite field result comes as an immediate consequence of their more general theorem.
The first power saving bound for the Furstenberg--S\'{a}rk\"{o}zy theorem in the function field setting is due to Green \cite{green} under the additional technical assumption that the number of roots of the polynomial $P$ is not divisible by $p$.} using the polynomial method of Croot--Lev--Pach \cite{clp}.

\begin{theorem}[{\cite[Corollary 1.5]{ls}}] \label{thm: ls}
	Let $p$ be a prime.
    Let $P(x) \in \F_p[x]$ be a polynomial of degree $d$ with $P(0) = 0$.
	There exists a positive constant $\gamma = \gamma(p, d) > 0$ such that if $k \in \N$ and $A \subseteq \F_{p^k}$ does not contain distinct $a, b \in A$ with $b - a = P(x)$ for some $x \in \F_{p^k}$, then $|A| \ll_{p,d} p^{k(1 - \gamma)}$.
\end{theorem}

One of the results of our paper is an improvement to the power saving bound in Theorem \ref{thm: ls} in the low characteristic context that matches the bound \eqref{eq: square root bound} from the high characteristic setting.
Namely, we show that the constant $\gamma$ can be taken equal to $\frac{1}{2}$, independently of the characteristic $p$ and the degree $d$ of the polynomial under consideration.

\begin{theorem} \label{thm: optimal power savings}
    Let $p$ be a prime.
	Let $P(x) \in \F_p[x]$ be a polynomial of degree $d$ with $P(0) = 0$.
	For any $k \in \N$, if $A \subseteq \F_{p^k}$ does not contain distinct $a, b \in A$ with $b - a = P(x)$ for some $x \in \F_{p^k}$, then
	\begin{equation*}
		|A| \ll_d p^{k/2}.
	\end{equation*}
\end{theorem}

\begin{remark}
	The exponent in Theorem \ref{thm: optimal power savings} is sharp.
	This follows from known bounds on the size of independent sets in generalized Paley graphs; see Proposition \ref{prop: sharpness}.
\end{remark}

The main tool in Theorem \ref{thm: optimal power savings} is an extension of the Weil bound to estimate exponential sums involving polynomials of arbitrary degree in low characteristic.

\begin{theorem} \label{thm: exponential sum}
    Let $p$ be a prime, let $P(x) \in \F_p[x]$ be a polynomial of degree $d$, and let $k \in \N$.
	If $\chi : \F_{p^k} \to \C$ is an additive character, then either
	\begin{equation*}
		\left| \sum_{x \in \F_{p^k}} \chi(P(x)) \right| = p^k \qquad \text{or} \qquad \left| \sum_{x \in \F_{p^k}} \chi(P(x)) \right| \le (d-1) p^{k/2}.
	\end{equation*}
\end{theorem}

\begin{remark}
    In the case $p \nmid d$ (in particular, if $d < p$), Theorem \ref{thm: exponential sum} is nothing but the classical Weil bound, and the only character $\chi$ for which $\left| \sum_{x \in \F_{p^k}} \chi(P(x)) \right| = p^k$ is the trivial character $\chi = 1$.
    Theorem \ref{thm: exponential sum} expands the scope of exponential sum estimates over finite fields by providing information about polynomials of arbitrary degree, with the necessary stipulation that there may be additional characters $\chi$ for which the exponential sum is as large as possible.
    It turns out that the collection of characters satisfying $\left| \sum_{x \in \F_{p^k}} \chi(P(x)) \right| = p^k$ may include many nontrivial characters but always has a nice algebraic description, which we provide in Theorem \ref{thm: exponential sum 2} below.
\end{remark}

Our approach using exponential sums has several advantages.
In addition to strengthening the power saving bound, the method has added flexibility that allows us to answer several other combinatorial questions about polynomial patterns over finite fields.
Consider, for example, the following two statements about a polynomial $P(x) \in \F_p[x]$:
\begin{itemize}
	\item	for any $\delta > 0$, there exists $K = K(P, \delta)$ such that if $k \ge K$ and $A \subseteq \F_{p^k}$ with $|A| \ge \delta p^k$, then there exist distinct $a, b \in A$ with $b - a = P(x)$ for some $x \in \F_{p^k}$
		(Furstenberg--S\'{a}rk\"{o}zy over finite fields);
	\item	for any $\delta > 0$, there exists $K = K(P, \delta)$ such that if $k \ge K$ and $A, B \subseteq \F_{p^k}$ with $|A| \cdot |B| \ge \delta p^{2k}$, then there exist $a \in A$ and $b \in B$ with $b - a = P(x)$ for some $x \in \F_{p^k}$
		(asymmetric Furstenberg--S\'{a}rk\"{o}zy over finite fields).
\end{itemize}
We give complete algebraic characterizations of the families of polynomials satisfying each of these statements and give a quantitative strengthening to the conclusion for the corresponding polynomials.

We also utilize a technique originating in ergodic theory \cite{b_density-schur, ert-update} to establish partition regularity of families of polynomial equations using the exponential sum estimate from Theorem \ref{thm: exponential sum}.


\subsection{Necessary and sufficient conditions for the Furstenberg--S\'{a}rk\"{o}zy theorem over finite fields}

The classical Furstenberg--S\'{a}rk\"{o}zy theorem was refined by Kamae and Mend\`{e}s France \cite{km}, who characterized the class of polynomials $P(x) \in \Z[x]$ for which every positive density subset of the integers contains a pair $a, b$ with $b - a = P(x)$ for some $x \in \Z$ as the family of polynomials with a root mod $m$ for every $m \in \N$ (so-called \emph{intersective polynomials}).
In the function field setting $\F_p[t]$, an analogous result holds, with the appropriate notion of intersective being that a polynomial $P(x) \in (\F_p[t])[x]$ has a root mod $g$ for every $g \in \F_p[t] \setminus \{0\}$.\footnote{This is essentially proved in \cite{bl} (see Theorem 9.2 and the remark following Theorem 9.5 therein).
However, there is a small error in the remark in \cite{bl}, which we briefly explain here. In the remark following Theorem 9.5 in \cite{bl}, intersective polynomials are defined as polynomials $P(x) \in (\F_p[t])[x]$ such that for any finite index subgroup $\Lambda \le (\F_p[t], +)$, there exists $m \in \F_p[t]$ such that $P(nm) \in \Lambda$ for every $n \in \Lambda$.
	The definition of intersective we have given is different and deals with a wider class of polynomials but is the correct notion to characterize the Furstenberg--S\'{a}rk\"{o}zy theorem in function fields.
	An example of an intersective polynomial that does not fit the condition in \cite{bl} is $P(x) = x+1$.
	Taking $\Lambda$ to be the subgroup $\Lambda = t\F_p[t]$ of index $p$, we have $P(nm) \equiv 1 \pmod{\Lambda}$ for every $n \in \Lambda, m \in \F_p[t]$, so the condition from \cite{bl} is not satisfied.
	However, $P(-1) = 0$, so $P$ is intersective (according to our definition).}
One may obtain $\F_{p^k}$ as a quotient of $\F_p[t]$, so any intersective polynomial $P(x) \in \F_p[x]$ will also satisfy the Furstenberg--S\'{a}rk\"{o}zy theorem over finite fields.
However, there are additional polynomials that satisfy the Furstenberg--S\'{a}rk\"{o}zy theorem over finite fields, so intersective is no longer the characterizing property.

In order to give a full description of polynomials satisfying the Furstenberg--S\'{a}rk\"{o}zy theorem over finite fields, we need a representation of a polynomial that is well-suited to algebraic manipulations in characteristic $p$.
There are two important classes of polynomials to consider when working in finite characteristic: \emph{separable} polynomials and \emph{additive} polynomials.

\begin{definition} \label{defn: seperable, additive}
    Let $p$ be a prime number.
    \begin{itemize}
        \item Call a monomial $x^d$ \emph{separable} (in characteristic $p$) if $p \nmid d$.
        \item A polynomial $P(x) = a_0 + \sum_{i=1}^n a_i x^{r_i} \in \F_p[x]$ is \emph{separable} if each nonconstant monomial $x^{r_i}$ is separable.
        \item We say that a polynomial $\eta(x) \in \F_p[x]$ is \emph{additive} if for any $k \in \N$ and any $x, y \in \F_{p^k}$, one has $\eta(x + y) = \eta(x) + \eta(y)$.
    \end{itemize}
\end{definition}

\begin{remark}
	The definition of additive polynomials involves looking at every finite field of characteristic $p$ for the following reason.
	If $\F_{p^k}$ is a fixed finite field, then the polynomial $x^{p^k}$ agrees (as a function on $\F_{p^k}$) with the polynomial $x$.
 As a consequence, there are many extra polynomials that behave additively as functions $\F_{p^k}$ but should not be considered as additive in characteristic $p$ in general.
	For example, the polynomial $P(x) = x^{2p^k} - x^2$ satisfies $P(x) = 0$ for $x \in \F_{p^k}$, so $P(x+y) = P(x) + P(y)$ for $x, y \in \F_{p^k}$.
	By considering $P$ as a function over a larger finite field such as $\F_{p^{k+1}}$, we can detect the non-additive behavior of $P$.
	
	One may equivalently define additive polynomials as those polynomial $\eta(x) \in \F_p[x]$ such that $\eta(x+y) = \eta(x) + \eta(y)$ for all $x,y \in \overline{\F}_p$, where $\overline{\F}_p$ is the algebraic closure of $\F_p$.
\end{remark}

Additive polynomials take the form $\eta(x) = \sum_{j=0}^m a_j x^{p^j}$.
Every polynomial $P(x) \in \F_p[x]$ has a unique representation as $P(x) = a_0 + \sum_{i=1}^n \eta_i(x^{r_i})$, where $\eta_1, \dots, \eta_n$ are nonzero additive polynomials and $x^{r_1}, \dots, x^{r_n}$ are distinct separable monomials.\footnote{Indeed, suppose $P(x) = a_0 + a_1 x + \dots + a_d x^d$.
For each $k \in \N$, we write $k = p^{j_k} s_k$ with $j_k \ge 0$ and $p \nmid s_k$.
Then $x^k = (x^{s_k})^{p^{j_k}}$, so $P(x) = a_0 + \sum_{i=1}^n \eta_i (x^{r_i})$, where $\{r _i : 1 \le i \le n\} = \{s_k : a_k \ne 0\}$ and $\eta_i(x) = \sum_{s_k = r_i} a_k x^{p^{j_k}}$.}
The crucial algebraic information is captured by the additive polynomials $\eta_1, \dots, \eta_n$, and we can encode all of this content in a single additive polynomial by the following lemma:

\begin{restatable}{lemma}{Euclidean} \label{lem: Euclidean}
	Let $\eta_1, \dots, \eta_n \in \F_p[x]$ be additive polynomials, let $H_i$ be the subgroup $H_i = \eta_i(\overline{\F}_p) \le (\overline{\F}_p, +)$ for $i = 1, \dots, n$, and let $H = \sum_{i=1}^n H_i$.
	There exists an additive polynomial $\eta \in \F_p[x]$ such that $\eta(\overline{\F}_p) = H$.
	Moreover, $\eta = \sum_{i=1}^n \eta_i \circ \zeta_i$ for some additive polynomials $\zeta_1, \dots, \zeta_n \in \F_p[x]$.
\end{restatable}

\begin{remark}
	The proof of Lemma \ref{lem: Euclidean} (given in Section \ref{sec: additive core}) is constructive and provides a simple algorithm for computing $\eta$ from $\eta_1, \dots, \eta_n$, so properties of $\eta$ are easily checkable for any given polynomial $P$.
 We use Lemma \ref{lem: Euclidean} as a crucial algebraic tool in proving many of the results of this paper.
\end{remark}

\begin{definition}
	Let $P(x) \in \F_p[x]$ be a nonconstant polynomial, and write $P(x) = a_0 + \sum_{i=1}^n \eta_i \left( x^{r_i} \right)$ with $\eta_i$ additive and $x^{r_i}$ separable and distinct.
	Let $\eta$ be an additive polynomial as produced by Lemma \ref{lem: Euclidean} from $\eta_1, \dots, \eta_n$.
	We call $\eta$ the \emph{additive core} of $P$.
\end{definition}

\begin{restatable}{theorem}{FFIntersective} \label{thm: FF intersective}
	Let $P(x) \in \F_p[x]$ be a nonconstant polynomial, let $\eta$ be its additive core.
	The following are equivalent:
	\begin{enumerate}[(i)]
		\item	for any $\delta > 0$, there exists $K = K(P, \delta)$ such that if $k \ge K$ and $A \subseteq \F_{p^k}$ with $|A| \ge \delta p^k$, then there exist distinct $a, b \in A$ with $b - a = P(x)$ for some $x \in \F_{p^k}$;
		\item	if $A \subseteq \F_{p^k}$ does not contain distinct $a, b \in A$ with $b - a = P(x)$ for some $x \in \F_{p^k}$, then $|A| \ll_d p^{k/2}$;
		\item	$a_0 = 0$ or $\eta(1) \ne 0$.
	\end{enumerate}
\end{restatable}

\begin{definition} \label{defn: FF-intersective}
    We call a polynomial satisfying any (all) of the conditions in Theorem \ref{thm: FF intersective} \emph{finite field intersective in characteristic $p$} (or \emph{FF$_p$-intersective} for short).
\end{definition}

We note that condition (iii) provides an efficient algorithmic method for checking if a polynomial is FF$_p$-intersective.\footnote{The problem of determining whether or not a polynomial in $\Z[x]$ or $(\F_p[t])[x]$ is intersective is decidable but less straightforward; see \cite[Theorem 1]{bb} for $\Z[x]$ and \cite[Theorem 1]{mishra} for a generalization to polynomials over rings of integers of global fields.}
Examples of FF$_p$-intersective polynomials include intersective polynomials (in the sense defined above that $P$ has a root mod $g$ for every $g \in \F_p[t] \setminus \{0\}$) and polynomials of degree $d < p$ (or, more generally, separable polynomials).
The simplest example of a non-FF$_p$-intersective polynomial is the polynomial $P(x) = x^p - x + 1$.


\subsection{Asymmetric Furstenberg--S\'{a}rk\"{o}zy theorem over finite fields}

Our next application of Theorem \ref{thm: exponential sum} is an asymmetric version of the Furstenberg--S\'{a}rk\"{o}zy theorem where we find elements $a$ and $b$ with $b - a = P(x)$ belonging to sets $A$ and $B$ that are allowed to differ from one another.
Such an enhancement is not possible in the integers due to the presence of ``local obstructions.''
In the finite field setting, an asymmetric enhancement is sometimes possible (for example if the polynomial has degree smaller than the characteristic) and is in other cases impossible (for example, if $P(x) = x^p - x$, then the group $H_k$ generated by the values of $P$ is a proper subgroup of $\F_{p^k}$, and one can take $A$ and $B$ to be distinct cosets of $H_k$).
We describe in Theorem \ref{thm: A,B Sarkozy} below the necessary and sufficient conditions for a polynomial $P$ to allow for an asymmetric form of the Furstenberg--S\'{a}rk\"{o}zy theorem.
The necessary and sufficient conditions involve the notion of \emph{equidistribution} for polynomial sequences in characteristic $p$, so we begin by introducing the basic definitions related to equidistribution that we will use.

\begin{definition} \label{defn: characters}~
    \begin{itemize}
        \item A character $\chi : \F_p[t] \to \C$ \emph{rational} (or \emph{periodic}) if there exists $f \in \F_p[t]$ such that $\chi(fg + h) = \chi(h)$ for every $g, h \in \F_p[t]$ and \emph{irrational} (\emph{aperiodic}) otherwise.
        \item A polynomial $P(x) \in \F_p[x]$ is \emph{good for irrational equidistribution} if
\begin{equation*}
	\lim_{N \to \infty} \E_{f \in \M_N} \chi(P(f)) = 0
\end{equation*}
for every irrational character $\chi \in \hat{\F_p[t]}$, where $\M_N = \{t^N + c_{N-1} t^{N-1} + \dots + c_1 t + c_0 : c_i \in \F_p\}$ is the family of monic polynomials of degree $N$ over $\F_p$.
    \end{itemize}
\end{definition}

In the following theorem, we fully characterize when a polynomial $P(x) \in \F_p[x]$ is good for irrational equidistribution in terms of a simple algebraic criterion.
Our proof (given in Section \ref{sec: irrational equidistribution}) combines a general Weyl-type equidistribution theorem from \cite{bl} with Lemma \ref{lem: Euclidean}.

\begin{restatable}{theorem}{IrrationalEquidistribution} \label{thm: irrational equidistribution}
	A polynomial $P(x) \in \F_p[x]$ is good for irrational equidistribution if and only if its additive core is of the form $\eta(x) = ax$ for some $a \in \F_p^{\times}$.
\end{restatable}

\begin{example} \label{eg: irrational equidistribution}
	(1) Every nonconstant separable polynomial (see Definition \ref{defn: seperable, additive}) is good for irrational equidistribution.
	(This was previously shown in \cite[Corollary 0.5]{bl}.)
	
	(2) The polynomial $P(x) = x^p$ is not good for irrational equidistribution.
	
	(3) More generally, an additive polynomial $P(x) = \sum_{j=0}^m{a_j x^{p^j}}$ is good for irrational equidistribution if and only if $P(x) = a_0 x$.
	
	(4) The polynomial $P(x) = x^{p^2} + x^{2p} - x$ is good for irrational equidistribution.
	Indeed, upon writing $P(x) = \eta_1(x) + \eta_2(x^2)$ with $\eta_1(x) = x^{p^2} - x$ and $\eta_2(x) = x^p$ and taking $\zeta_1(x) = -x$ and $\zeta_2(x) = x^p$, we see that $(\eta_1 \circ \zeta_1 + \eta_2 \circ \zeta_2)(x) = x$.
	
	(5) The polynomial $P(x) = x^{2p} - x^2$ is not good for irrational equidistribution, as can be seen by expressing $P(x) = \eta(x^2)$ with $\eta(x) = x^p - x$.
\end{example}

There is one additional observation that we should make before stating our asymmetric version of the Furstenberg--S\'{a}rk\"{o}zy theorem over finite fields:
since the Frobenius map $\Phi : x \mapsto x^p$ is an automorphism of $\F_{p^k}$, the polynomials $P$ and $P \circ \Phi$ have the same image in $\F_{p^k}$.
Up to this trivial modification, we show that being good for irrational equidistribution is a necessary and sufficient condition for an asymmetric Furstenberg--S\'{a}rk\"{o}zy theorem:

\begin{restatable}{theorem}{ABSarkozy} \label{thm: A,B Sarkozy}
	Let $P(x) \in \F_p[x]$.
	The following are equivalent:
	\begin{enumerate}[(i)]
		\item	there exists a polynomial $Q(x) \in \F_p[x]$ and an integer $s \ge 0$ such that $Q$ is good for irrational equidistribution and $P(x) = Q \left( x^{p^s} \right)$;
		\item	for any $\delta > 0$, there exists $K_1 = K_1(P, \delta)$ such that if $k \ge K_1$ and $A, B \subseteq \F_{p^k}$ satisfy $|A| \cdot |B| \ge \delta p^{2k}$, then there exists $a \in A$ and $b \in B$ with $b - a = P(x)$ for some $x \in \F_{p^k}$;
		\item	for any $\delta > 0$, there exists $K_2 = K_2(P, \delta)$ such that if $k \ge K_2$ and $A, B \subseteq \F_{p^k}$ satisfy $|A| \cdot |B| \ge \delta p^{2k}$, then $A + B + S = \F_{p^k}$, where $S = P(\F_{p^k})$;
		\item	for any $A, B \subseteq \F_{p^k}$;
			\begin{equation*}
				\left| \left\{ (x,y) \in \F_{p^k} \times \F_{p^k} : x \in A~\text{and}~x + P(y) \in B \right\} \right| = |A| |B| + O \left( p^{k/2} \sqrt{|A| |B|} \right).
			\end{equation*}
	\end{enumerate}
\end{restatable}

\begin{remark}
	(1) Note that by Theorem \ref{thm: irrational equidistribution}, (i) is equivalent to the condition $\sum_{i=1}^n \eta_i \circ \zeta_i(x) = a x^{p^s}$ for some additive polynomials $\zeta_1, \dots, \zeta_n$ and $a \in \F_p^{\times}$, where $P(x) = a_0 + \sum_{i=1}^n \eta_i(x^{r_i})$ is the representation of $P$ in terms of additive polynomials $\eta_i$ and distinct separable monomials $x^{r_i}$.
	Using the algorithmic method behind Lemma \ref{lem: Euclidean}, one can therefore check by hand whether or not a polynomial satisfies condition (i).
	
	(2) At first glance, one may be tempted to explain the phenomenon $A + B + S = \F_{p^k}$ in item (iii) by the fact that $S$ is a large subset of $\F_{p^k}$ (it has density at least $d^{-1}$, where $d = \deg{P}$).
	However, this is too naive an explanation: if $P$ does not satisfy (i), then we meet an algebraic obstruction that allows for large subsets $A, B \subseteq \F_{p^k}$ with $A + B + S \ne \F_{p^k}$.
 This algebraic obstruction can be seen explicitly in the proof of Theorem \ref{thm: A,B Sarkozy} in Section \ref{sec: A,B Sarkozy}.
\end{remark}


\subsection{Partition regular polynomial equations over finite fields}

Our last application of Theorem \ref{thm: exponential sum} concerns partition regularity of polynomial equations.
Combining Theorem \ref{thm: exponential sum} with the technology of Loeb measures on ultraproduct spaces, we are able to establish partition regularity of families of polynomial equations over finite fields, such as the following:

\begin{restatable}{theorem}{PartitionRegularity} \label{thm: P-P=Q}
	Let $P(x) \in \F_p[x]$ be a nonconstant polynomial, and let $Q(x) \in \F_p[x]$ be FF$_p$-intersective (see Definition \ref{defn: FF-intersective}).
	For any $r \in \N$, there exists $K = K(P, Q, r) \in \N$ and $c = c(P, Q, r) > 0$ such that for any $k \ge K$ and any $r$-coloring $\F_{p^k} = \bigcup_{i=1}^r{C_i}$, there are at least $c p^{2k}$ monochromatic solutions to the equation $P(x) - P(y) = Q(z)$.
	That is,
	\begin{equation*}
		\left| \left\{ (x,y,z) \in \F_{p^k}^3 : P(x) - P(y) = Q(z)~\text{and}
		 ~\{x,y,z\} \subseteq C_i~\text{for some}~i \in \{1, \dots, r\} \right\} \right| \ge c p^{2k}
	\end{equation*}
\end{restatable}

Since any polynomial with zero constant term is FF$_p$-intersective for every prime $p$, one application of note is a polynomial Schur theorem over finite fields:

\begin{corollary} \label{cor: polynomial Schur}
	Let $P(x) \in \Z[x]$ with $P(0) = 0$.
	Then for any $r \in \N$, there exists $N = N(P,r) \in \N$ and $c = c(P,r) > 0$ such that if $q > N$ and $\F_q = \bigcup_{i=1}^r{C_i}$, then there are at least $cq^2$ monochromatic solutions to the equation $P(x) + P(y) = P(z)$.
	In particular, if the coefficients of $P$ are not all divisible by the characteristic of $\F_q$, then there are $\gg_{P,r} q^2$ monochromatic solutions with $P(x), P(y), P(z) \ne 0$.
\end{corollary}

\begin{remark}
	In the case $P(x) = x^d$, Corollary \ref{cor: polynomial Schur} corresponds to the Fermat equation $x^d + y^d = z^d$.
 The easier problem (in comparison to partition regularity) of proving existence of solutions to the Fermat equation over finite fields has a long history and inspired many substantial developments in number theory.
 One fruitful point of view is to see the equation $x^d + y^d = z^d$ as an instance of a \emph{diagonal equation}, a family of polynomial equations dealt with systematically by Weil and for which very precise estimates on the number of solutions over finite fields can be obtained using exponential sums; see \cite{weil}.
 Earlier contributions to the problem of finding solutions to the Fermat equation over finite fields include those of Dickson, who dealt with special cases over prime fields using exponential sums in \cite{dickson1, dickson2}, and Schur, who proved existence of solutions (though without strong estimates on the number of solutions) over large prime fields using his eponymous partition regularity theorem in \cite{schur}.
 
	The much stronger property of partition regularity of the Fermat equation was established previously in the context of prime fields in \cite[Theorem 4]{cgs} and generalized to a family of related polynomial equations in \cite{lindqvist}.
	We complete the picture here by extending the partition regularity property to arbitrary finite fields of sufficiently large order (with no assumption on the characteristic).
	
	Related density results for Pythagorean pairs and triples in finite fields were obtained in \cite[Section 6]{dlms}, where the authors also show that a density version (``density regularity'') of Corollary \ref{cor: polynomial Schur} fails already for the Pythagorean equation $x^2 + y^2 = z^2$.
\end{remark}


\subsection{Quasi-randomness and asymptotic total ergodicity}

The results of this paper can be placed in a broader context, linking the combinatorial phenomenon of \emph{quasi-randomness} and the dynamical phenomenon of (asymptotic) \emph{total ergodicity}.
If one is interested in finding configurations of the form $\{x, x + P(y)\}$ in large subsets of a ring $R$, a natural combinatorial object to consider is the Cayley graph with vertex set $R$ and edges $E = \{\{a,b\} : b - a = P(x)~\text{for some}~x \in R, P(x) \ne 0\}$.
The independent sets in this graph correspond to subsets of $R$ avoiding configurations of the form $\{x, x+P(y)\}$.
Taking $R = \Z$ brings us to the setting of the classical Furstenberg--S\'{a}rk\"{o}zy theorem, and taking $R$ to be a finite field brings us to the setting of the present paper.
The strength of the bounds in Theorem \ref{thm: optimal power savings} and the availability of asymmetric forms of the Furstenberg--S\'{a}rk\"{o}zy theorem over finite fields (as in Theorem \ref{thm: A,B Sarkozy}) can be linked to the phenomenon of quasi-randomness, as we explain below.

For the sake of the present discussion, let us consider the polynomial $P(x) = x^2$.
The Cayley graph for $(\F_q, +)$ generated by the squares is called the \emph{Paley graph of order $q$}, named after the mathematician Raymond Edward Alan Christopher Paley for his construction of Hadamard matrices using properties of quadratic residues over finite fields \cite{paley}.\footnote{The complete story of how a family of graphs came to bear Paley's name is rather complicated and does not seem to be fully known. Paley's 1933 paper \cite{paley} did not involve any graphs, nor did subsequent work on Hadamard matrices by his contemporaries (e.g. \cite{todd, coxeter}). The graphs now known as Paley graphs were first defined independently by Sachs \cite{sachs} and by Erd\H{o}s and R\'{e}nyi \cite{er} in the early 1960s, but no name was assigned to the family of graphs in their papers. By the 1970s, the term ``Paley graph'' had become standard and appeared in the book of Cameron and van Lindt \cite{cvl} in 1975 without any explanation regarding the source of the name. Gareth A. Jones has documented much of the history of Paley graphs and their attribution, and we invite the reader to explore his paper \cite{jones_Paley}, from which we have drawn our summary here.}
To be precise, the Paley graph of order $q$ is the graph $\bfP_q$ with vertex set $\F_q$ and edges $\{a,b\}$ if and only if $b - a$ is a nonzero square.
(One typically assumes $q \equiv 1 \pmod{4}$ so that $b-a$ is a square if and only if $a-b$ is a square.)
The family of Paley graphs is an example of a \emph{quasi-random} family.
A sequence of graphs $G_n = (V_n, E_n)$ with $N_n$ vertices and edge density $p_n = |E_n|/\binom{N_n}{2}$ is \emph{quasi-random} if for every $n \in \N$ and every pair of subsets $A, B \subseteq V_n$,
\begin{equation*}
	\left| \left\{ \{a,b\} \in E_n : a \in A, b \in B \right\} \right| = p_n |A| |B| + o(N_n^2).
\end{equation*}
Quasi-random graphs were introduced by Chung, Graham, and Wilson in \cite{cgw}, where the authors provided several equivalent characterizations of quasi-randomness and proved that Paley graphs are quasi-random.

The quasi-randomness of the family of Paley graphs is equivalent to the estimate
\begin{equation} \label{eq: Paley quasi-randomness}
	\left| \left\{ (x,y) \in \F_q \times \F_q : x \in A, x + y^2 \in B \right\} \right| = |A| |B| + o(q^2)
\end{equation}
for $A, B \subseteq \F_q$, since $\bfP_q$ has edge density $p = \frac{1}{2}$ and the quantity on the left hand side of \eqref{eq: Paley quasi-randomness} counts each edge between $A$ and $B$ twice.
Item (iv) of Theorem \ref{thm: A,B Sarkozy} can thus be seen as a generalization of \eqref{eq: Paley quasi-randomness}, establishing a connection between irrational equidistribution (via property (i) in Theorem \ref{thm: A,B Sarkozy}) and quasi-randomness.
Another simple consequence of quasi-randomness is that quasi-random graphs cannot have large independent sets (see, e.g., \cite[Proposition 4.5]{ks}), which leads to the strong power-saving bounds as in Theorem \ref{thm: optimal power savings}.

Some of the above-described combinatorial results in the finite field setting (in particular, an asymmetric form of the Furstenberg--S\'{a}rk\"{o}zy theorem and power-saving bounds for several variations of the Furstenberg--S\'{a}rk\"{o}zy theorem) do not have natural analogues in the integers.
One may ask: from the point of view of quasi-randomness, what is the essential difference between the integers and a finite field?
The answer to this question hinges on a surprising connection to dynamics.
In a companion paper \cite{ab_total-ergodicity}, we show that quasi-randomness of generalized Paley graphs associated with a sequence of finite commutative rings $(R_n)_{n \in \N}$ is closely related to \emph{asymptotic total ergodicity} of the sequence of rings.\footnote{We do not give a full definition of asymptotic total ergodicity here, as it would take us too far astray.
The notion of asymptotic total ergodicity comes as a finitization of the phenomenon of total ergodicity in ergodic theory and was previously defined for modular rings in \cite{bb_TE}.}
We also establish extensions of Theorems \ref{thm: FF intersective} and \ref{thm: A,B Sarkozy} for asymptotically totally ergodic sequences of rings as manifestations of quasi-randomness.
We refer the reader to \cite{ab_total-ergodicity} for more details and for precise statements of the results alluded to here.


\subsection{Outline of the paper}

We prove the main algebraic lemma, Lemma \ref{lem: Euclidean}, in Section \ref{sec: additive core}.
The main exponential sum estimate of the paper (Theorem \ref{thm: exponential sum}) is proved in Section \ref{sec: Weil}.
The remaining four sections address combinatorial applications.
We prove a power saving bound for the Furstenberg--S\'{a}rk\"{o}zy theorem over finite fields (Theorem \ref{thm: optimal power savings}) and provide necessary and sufficient conditions for a polynomial to satisfy the Furstenberg--S\'{a}rk\"{o}zy theorem over finite fields (Theorem \ref{thm: FF intersective}) in Section \ref{sec: power saving}.
In Section \ref{sec: irrational equidistribution}, we prove Theorem \ref{thm: irrational equidistribution} as a crucial ingredient for proving necessary and sufficient conditions for an asymmetric form of the Furstenberg--S\'{a}rk\"{o}zy theorem over finite fields (Theorem \ref{thm: A,B Sarkozy}) in Section \ref{sec: A,B Sarkozy}.
The final section, Section \ref{sec: partition regularity}, is concerned with partition regularity of polynomial equations.


\section{Additive core of polynomials over $\F_p$} \label{sec: additive core}

In this short section, we prove Lemma \ref{lem: Euclidean}, which will serve as an important algebraic tool for several of the later results of the paper.
Recall the statement of Lemma \ref{lem: Euclidean}:

\Euclidean*

\begin{proof}
	It suffices to prove the $n = 2$ case, since the general case easily follows by induction.
	
	If $\eta_i = 0$ for some $i \in \{1, 2\}$, then take $\eta = \eta_j$ with $j \ne i$.
	
	Suppose now that $\eta_1$ and $\eta_2$ are both nonzero.
	Write $\eta_1(x) = \sum_{i=0}^m{a_i x^{p^i}}$ and $\eta_2(x) = \sum_{j=0}^l{b_j x^{p^j}}$.
	Without loss of generality, $m \ge l$.
	Define
	\begin{equation} \label{eq: degree reduction}
		\eta'_1(x) = b_l \eta_1(x) - a_m \eta_2 \left( x^{p^{m-l}} \right)
		 = \eta_1(b_lx) + \eta_2 \left( -a_m x^{p^{m-l}} \right),
	\end{equation}
	and let $H'_1 = \eta'_1(\overline{\F}_p)$.
	Then $\deg{\eta'_1} < \deg{\eta_1}$. \\
	
	\underline{Claim}: $H'_1 + H_2 = H_1 + H_2$.
	
	Since $H_1, H_2$, and $H'_1$ are all subgroups of $\overline{\F}_p$, it suffices to show that $H'_1 \subseteq H_1 + H_2$ and $H_1 \subseteq H'_1 + H_2$.
	For any $x \in \overline{\F}_p$, \eqref{eq: degree reduction} expresses $\eta'_1(x)$ as a sum of an element of $H_1$ and an element of $H_2$.
	Hence, $H'_1 \subseteq H_1 + H_2$.
	Rearranging \eqref{eq: degree reduction}, we have
	\begin{equation*}
		\eta_1(x) = b_l^{-1} \eta'_1(x) + b_l^{-1} a_m \eta_2 \left( x^{p^{m-l}} \right).
	\end{equation*}
	Thus, $H_1 \subseteq H'_1 + H_2$.
	This proves the claim. \\
	
	We have shown that, given any nonzero additive polynomials $\eta_1, \eta_2 \in \F_p[x]$, we may find $\eta'_1, \eta'_2 \in \F_p[x]$ with $\eta'_1(\overline{\F}_p) + \eta'_2(\overline{\F}_p) = \eta_1(\overline{\F}_p) + \eta_2(\overline{\F}_p)$ such that $\deg{\eta'_1} + \deg{\eta'_2} < \deg{\eta_1} + \deg{\eta_2}$, and $\eta'_1$ and $\eta'_2$ are of the appropriate form.
	Repeating this process finitely many times, we eventually reduce to the situation that one of the additively polynomials is zero.
	We then take $\eta$ to be the remaining nonzero polynomial.
\end{proof}

\begin{remark} \label{rem: additive core via other rings}
    Let $\eta_1, \ldots, \eta_n$ and $\eta$ be as in Lemma \ref{lem: Euclidean}.
    For a given commutative ring $R$ of characteristic $p$, let $H_i(R) = \eta_i(R) \le (R,+)$ and $H(R) = \sum_{i=1}^n H_i(R)$.
    Following the same argument as in the proof of Lemma \ref{lem: Euclidean} above, we have $\eta(R) = H(R)$ for every commutative ring $R$ of characteristic $p$.
    In fact, if $I \le R$ is an ideal, then we also have $\eta(I) = H(I)$.
    In particular, $\eta(g\F_p[t]) = \sum_{i=1}^n \eta_i(g\F_p[t])$ for every $g \in \F_p[t]$.
\end{remark}

The argument in the proof of Lemma \ref{lem: Euclidean} provides an algorithm for obtaining $\eta$ that bears a strong resemblance with the Euclidean algorithm.
We work through a few simple examples to see more concretely how the algorithm works.

\begin{example} \label{eg: Euclidean}
	(1) $\eta_1(x) = x^{p^2} - x$, $\eta_2(x) = x^{p^3} + x^p$.
	The polynomial $\eta_2$ has larger degree, so we shift the exponents of $\eta_1$ to match the degree of $\eta_2$ and subtract:
	\begin{equation*}
		\eta'_2(x) = \eta_2(x) - \eta_1(x^p) = 2x^p.
	\end{equation*}
	If $p = 2$, then $\eta'_2(x) = 0$, so we stop, and the resulting polynomial $\eta$ is simply $\eta_1$.
	(Note that when $p = 2$, $\eta_1$ may be rewritten as $\eta_1(x) = x^{p^2} + x$, and then it is clear that $\eta_2(x) = \eta_1(x^p)$, so the image of $\eta_2$ is manifestly a subset of the image of $\eta_1$.)
	Suppose $p > 2$.
	Then $\deg{\eta_1} > \deg{\eta'_2}$, so we shift the exponents of $\eta'_2$ and subtract:
	\begin{equation*}
		\eta'_1(x) = 2 \eta_1(x) - \eta'_2(x^p) = -2x.
	\end{equation*}
	Since $p > 2$, the element $-2 \in \F_p$ is invertible, so the image of $\eta'_1$ is all of $\overline{\F}_p$, and we are done: $\eta(x) = \eta'_1(x) = -2x$.
	(One can check that applying one more step of the algorithm would result in $\eta''_2 = 0$, indicating that the process has terminated.)
	
	(2) $\eta_1(x) = x^{p^3} + x^{p^2} + x^p$, $\eta_2(x) = x^{p^2}$.
	First, shifting $\eta_2$ and subtracting, we have
	\begin{equation*}
		\eta'_1(x) = \eta_1(x) - \eta_2(x^p) = x^{p^2} + x^p.
	\end{equation*}
	Next, subtracting $\eta_2$ without any shifting gives
	\begin{equation*}
		\eta''_1(x) = \eta'_1(x) - \eta_2(x) = x^p.
	\end{equation*}
	Shifting $\eta''_1$ and subtracting from $\eta_2$ produces $\eta'_2 = 0$, so we are done and $\eta(x) = \eta''_1(x) = x^p$.
\end{example}


\section{Exponential sum bound} \label{sec: Weil}

The goal of this section is to prove the exponential sum bound (Theorem \ref{thm: exponential sum}). As preparation, we recall basic notions from Fourier analysis on finite fields.

Fix a prime $p$ and $k \in \N$.
The \emph{trace} $\textup{Tr} : \F_{p^k} \to \F_p$ is the $\F_p$-linear map $\textup{Tr}(x) = x + x^p + \dots + x^{p^{k-1}}$.
Let $e_{p^k} : \F_{p^k} \to \C$ be the group homomorphism $e_{p^k}(x) = \exp \left( \frac{2 \pi i \cdot \textup{Tr}(x)}{p} \right)$.
When it is clear from context, we will drop the subscript and simply write $e$ for the function $e_{p^k}$.
Additive characters on $\F_{p^k}$ take the form $x \mapsto e(\xi x)$ for $\xi \in \F_{p^k}$; see \cite[Proposition 1.13]{kowalski}.

Using this isomorphism between $\F_{p^k}$ and it dual $\hat{\F}_{p^k}$, we define the \emph{Fourier transform} of a function $f : \F_{p^k} \to \C$ to be the function $\hat{f} : \F_{p^k} \to \C$ given by
\begin{equation*}
	\hat{f}(\xi) = \E_{x \in \F_{p^k}} f(x) e(-\xi x).
\end{equation*}
The Fourier transform has the following basic properties:
\begin{itemize}
	\item	Fourier inversion formula:
		\begin{equation*}
			f(x) = \sum_{\xi \in \F_{p^k}}{\hat{f}(\xi) e(\xi x)}
		\end{equation*}
	\item	Parseval's identity:
		\begin{equation*}
			\E_{x \in \F_{p^k}}{\left| f(x) \right|^2} = \sum_{\xi \in \F_{p^k}}{\left| \hat{f}(\xi) \right|^2}.
		\end{equation*}
\end{itemize}

With the notation above, we now recall the Weil bound:

\begin{theorem}[Weil bound, cf. \cite{kowalski}, Theorem 3.2] \label{thm: Weil}
	Let $q$ be any prime power.
	Let $P(x) \in \F_q[x]$ be a polynomial of degree $d$.
	If $d < q$ and $\gcd(d,q) = 1$, then for any $\xi \in \F_q \setminus \{0\}$, one has
	\begin{equation*}
		\left| \E_{x \in \F_q} e \left( \xi P(x) \right) \right| \le (d-1) q^{-1/2}
	\end{equation*}
\end{theorem}

To prove Theorem \ref{thm: exponential sum}, we will combine the Weil bound with algebraic information about a polynomial encoded in its additive core.
This immediately leads to a stronger version of Theorem \ref{thm: exponential sum} that gives additional information about when the character sum is nontrivial:

\begin{theorem} \label{thm: exponential sum 2}
	Let $P(x) \in \F_p[x]$ be a polynomial of degree $d$.
	Let $\eta$ be the additive core of $P$, and let $a_0 = P(0)$.
	Then for any $k \in \N$,
	\begin{enumerate}[(1)]
		\item	$H_k = \eta(\F_{p^k})$ is the group generated by $\{P(x) - a_0 : x \in \F_{p^k}\}$, and
		\item	for any $\xi \in \F_{p^k}$,
			\begin{equation*}
				\left| \E_{x \in \F_{p^k}} e(\xi P(x)) - e(\xi a_0) \ind_{H_k^{\perp}}(\xi) \right| \le (d-1) p^{-k/2}.
			\end{equation*}
	\end{enumerate}
\end{theorem}

\begin{proof}[Proof of Theorem \ref{thm: exponential sum} assuming Theorem \ref{thm: exponential sum 2}]
	Let $\chi : \F_{p^k} \to \C$ be an additive character.
	Write $\chi(x) = e(\xi x)$ for some $\xi \in \F_{p^k}$.
	
	If $\xi \in H_k^{\perp}$, then since $P(x) - a_0 \in H_k$ by (1) in Theorem \ref{thm: exponential sum 2}, we have $e(\xi P(x)) = e(\xi a_0)$ for $x \in \F_{p^k}$.
	Hence,
	\begin{equation*}
		\left| \sum_{x \in \F_{p^k}} \chi(P(x)) \right| = \left| p^k e(\xi a_0) \right| = p^k.
	\end{equation*}
	
	If $\xi \notin H_k^{\perp}$, then by (2) in Theorem \ref{thm: exponential sum 2}, we have
	\begin{equation*}
		\left| \sum_{x \in \F_{p^k}} \chi(P(x)) \right| = \left| p^k \E_{x \in \F_{p^k}} e(\xi P(x)) \right| \le (d-1) p^{k/2}.
	\end{equation*}
\end{proof}

\begin{proof}[Proof of Theorem \ref{thm: exponential sum 2}]
	Write $P(x) = a_0 + \sum_{i=1}^n \eta_i(x^{r_i})$ with $\eta_i$ additive and $x^{r_i}$ distinct and separable.
	Let $H_{k,i} = \eta_i(\F_{p^k})$.
	Then by the definition of the additive core $\eta$, we have $H_k = \sum_{i=1}^n H_{k,i}$, so clearly $P(x) - a_0 \in H_k$.
	If the group $\left\langle P(x) - a_0 : x \in \F_{p^k} \right\rangle$ is a proper subgroup of $H_k$, then $H_k^{\perp} \subsetneq \left\langle P(x) - a_0 : x \in \F_{p^k} \right\rangle^{\perp}$, so there is exists $\xi \in \F_{p^k}$ such that $e(\xi P(x)) = e(\xi a_0)$ for every $x \in \F_{p^k}$ but $\xi \notin H_k^{\perp}$.
	Therefore, (1) follows from (2), so we will prove (2) directly.
	
	Fix $\xi \in \F_{p^k}$.
	If $\xi \in H_k^{\perp}$, then $e(\xi P(x)) = e(\xi a_0)$ for every $x \in \F_{p^k}$, so
	\begin{equation*}
		\left| \E_{x \in \F_{p^k}} e(\xi P(x)) - e(\xi a_0) \ind_{H_k^{\perp}}(\xi) \right| = 0.
	\end{equation*}
	Suppose $\xi \notin H_k^{\perp}$.
	For each $i \in \{1, \dots, n\}$, the map $x \mapsto e(\xi \eta_i(x))$ is again an additive character on $\F_{p^k}$, so there exists $c_i \in \F_{p^k}$ such that $e(\xi \eta_i(x)) = e(c_i x)$.
	Thus,
	\begin{equation*}
		e(\xi P(x)) = e(\xi a_0) e \left( \sum_{i=1}^n{c_i x^{r_i}} \right).
	\end{equation*}
	Since $H_k^{\perp} = \bigcap_{i=1}^n H_{k,i}^{\perp}$, we have $c_i \ne 0$ for some $i \in \{1, \dots, n\}$.
	Moreover, $p \nmid r_i$ for $i \in \{1, \dots, n\}$, so
	\begin{equation*}
		\left| \sum_{x \in \F_{p^k}} e(\xi P(x)) \right| = \left| \sum_{x \in \F_{p^k}} e \left( \sum_{i=1}^n{c_i x^{r_i}} \right) \right| \le (d-1) p^{k/2}
	\end{equation*}
	by Theorem \ref{thm: Weil}.
\end{proof}

\begin{corollary} \label{cor: projection}
	Let $P(x) \in \F_p[x]$ be a nonconstant polynomial, and let $\eta$ be its additive core.
	Then for any $k \in \N$ and any $f : \F_{p^k} \to \C$,
	\begin{equation*}
		\norm{L^2(\F_{p^k})}{\E_{y \in \F_{p^k}} f(x + P(y)) - \E_{z \in H_k} f(x + a_0 + z)} \le (d-1) p^{-k/2} \norm{L^2(\F_{p^k})}{f},
	\end{equation*}
	where $H_k = \eta(\F_{p^k})$ and $a_0 = P(0)$.
\end{corollary}

\begin{proof}
	Let $F(x) = \E_{y \in \F_{p^k}} f(x + P(y)) - \E_{z \in H_k} f(x + a_0 + z)$.
	Then by direct calculation,
	\begin{multline*}
		\hat{F}(\xi) = \E_{x \in \F_{p^k}} F(x) e(-\xi x) \\
		= \E_{x \in \F_{p^k}}  \E_{y \in \F_{p^k}} f(x) e(- \xi x) e(\xi P(y)) - \E_{x \in \F_{p^k}} \E_{z \in H_k} f(x) e(-\xi x) e(\xi a_0) e(\xi z) \\
		 = \hat{f}(\xi) \left( \E_{y \in \F_{p^k}} e(\xi P(y)) - e(\xi a_0) \ind_{H_k^{\perp}}(\xi) \right).
	\end{multline*}
	Therefore, by Theorem \ref{thm: exponential sum 2},
	\begin{equation*}
		\left| \hat{F}(\xi) \right| \le (d-1) p^{-k/2} \left| \hat{f}(\xi) \right|.
	\end{equation*}
	Thus, by Parseval's identity, we have
	\begin{equation*}
		\norm{L^2(\F_{p^k})}{F} \le (d-1) p^{k/2} \left( \sum_{\xi \in \F_{p^k}} \left| \hat{f}(\xi) \right|^2 \right)^{1/2} = (d-1) p^{-k/2} \norm{L^2(\F_{p^k})}{f}.
	\end{equation*}
\end{proof}


\section{Power saving bound for the Furstenberg--S\'{a}rk\"{o}zy theorem in characteristic $p$} \label{sec: power saving}

Our first combinatorial application of Theorem \ref{thm: exponential sum} is a power-saving bound for the Furstenberg--S\'{a}rk\"{o}zy theorem over finite fields of characteristic $p$.
Theorem \ref{thm: optimal power savings}, which deals with polynomials with zero constant term, is a special case of Theorem \ref{thm: FF intersective}, so we will only prove Theorem \ref{thm: FF intersective}, restated below for convenience:

\FFIntersective*

\begin{proof}[Proof of Theorem \ref{thm: FF intersective}]
	Consider the additional statement
	\begin{enumerate}
		\item[(iv)]	the group $\langle P(x) - a_0 : x \in \F_{p^k} \rangle \le (\F_{p^k}, +)$ contains $a_0$ for all large $k \in \N$.
	\end{enumerate}
	First we will show that items (i), (ii), and (iv) are equivalent. \\
	
	(i) $\implies$ (iv).
	We will prove the contrapositive.
	Suppose (iv) fails.
	Let $k_i$ be an increasing sequence such that $a_0 \notin \langle P(x) - a_0 : x \in \F_{p^{k_i}} \rangle$.
	Let $A = \langle P(x) - a_0 : x \in \F_{p^{k_i}} \rangle \subseteq \F_{p^{k_i}}$, and note that $|A| \ge \left| P \left( \F_{p^{k_i}} \right) \right| \ge \frac{p^{k_i}}{d}$.
	For any $a, b \in A$, we have $b - a \in A$.
	On the other hand, for any $x \in \F_{p^k}$, we have $P(x) \in a_0 + A$, so $A$ does not contain $a, b$ with $b - a = P(x)$.
	Thus, (i) fails for $\delta = \frac{1}{d}$. \\
	
	(iv) $\implies$ (ii).
	Let $A \subseteq \F_{p^k}$ and suppose $A$ does not contain distinct $a, b \in A$ with $b - a = P(x)$ for some $x \in \F_{p^k}$.
	Then
	\begin{equation*}
		\Lambda(A) = \E_{x,y \in \F_{p^k}} \ind_A(x) \ind_A(x + P(y)) = \frac{\left| \left\{ y \in \F_{p^k} : P(y) = 0 \right\} \right|}{p^k} \frac{|A|}{p^k} \le d p^{-2k} |A|.
	\end{equation*}
	However, by Corollary \ref{cor: projection} and the Cauchy--Schwarz inequality,
	\begin{equation*}
		\left| \Lambda(A) - \E_{x \in \F_{p^k}, z \in H_k} \ind_A(x) \ind_A(x + a_0 + z) \right| \le (d-1) p^{-k/2} \norm{L^2(\F_{p^k})}{\ind_A}^2 = (d-1) p^{-3k/2} |A|
	\end{equation*}
	and since $a_0 \in H$, we have
	\begin{equation*}
		\E_{x \in \F_{p^k}, z \in H_k} \ind_A(x) \ind_A(x + a_0 + z) = \E_{x \in \F_{p^k}, z \in H_k} \ind_A(x) \ind_A(x + z) \ge p^{-2k} |A|^2.
	\end{equation*}
	Thus, $p^{-2k} |A|^2 \le dp^{-2k}|A| + (d-1) p^{-3k/2}|A|$, which after rearranging results in $|A| \ll_d p^{k/2}$. \\
	
	(ii) $\implies$ (i) is trivial. \\
	
	In order to prove the equivalence between the two algebraic conditions (iii) and (iv), we first make a couple of observations.
	If $a_0 = 0$, then (iii) and (iv) both hold, so we will assume $a_0 \ne 0$.
	Now, the group $H_k := \langle P(x) - a_0 : x \in \F_{p^k} \rangle \le (\F_{p^k}, +)$ is equal to $\eta(\F_{p^k})$ by Theorem \ref{thm: exponential sum 2}(1).
	Moreover, $\eta$ is $\F_p$-linear, so $H_k$ contains $a_0$ if and only if $\eta - 1$ has a root in $\F_{p^k}$.
	It therefore suffices to prove $\eta - 1$ has a root in $\F_{p^k}$ for all large $k \in \N$ if and only if $\eta(1) \ne 0$.
	
	Suppose $c = \eta(1) \ne 0$.
	Then since $\eta$ is $\F_p$-linear, we have $\eta(c^{-1}) = c^{-1} \eta(1) = 1$, so $c^{-1}$ is a root of $\eta - 1$.
	
	Conversely, let $P = \eta - 1$, and suppose $R = \left\{ k \in \N : P~\text{has a root in}~\F_{p^k} \right\}$ is cofinite.
	Note that we can equivalently express $R$ as the set of $k \in \N$ for which $\gcd(P, x^{p^k}-x) \ne 1$, since $x^{p^k} - x = 0$ for $x \in \F_{p^k}$.
	The polynomials $Q_k(x) = x^{p^k} - x$ have the property $\gcd(Q_k, Q_l) = Q_{\gcd(k,l)}$.
	In particular, if $q_1, q_2 \in \P$ are distinct prime numbers, then $\gcd(Q_{q_1}, Q_{q_2}) = x^p - x$.
	Since $P$ has only finitely many irreducible factors, $\gcd(P, Q_k)$ takes only finitely many values, so by the pigeonhole principle, there is a nonconstant polynomial $D(x) \in \F_p[x]$ such that the set $\{k \in R \cap \P : \gcd(P, Q_k) = D\}$ is infinite.
	But then $D \mid Q_q$ for infinitely many $q \in \P$, which implies $D \mid x^p - x$.
	Thus, $\gcd(P, x^p - x) \ne 1$.
	Equivalently, $P$ has a root in $\F_p$, say $P(c) = 0$.
	Then $\eta(1) = c^{-1} \eta(c) = c^{-1} (P(c)+1) = c^{-1} \ne 0$.
\end{proof}

Properly interpreting known bounds on parameters of generalized Paley graphs gives a complementary lower bound, showing that the exponent in item (ii) in Theorem \ref{thm: FF intersective} cannot be improved.

\begin{proposition} \label{prop: sharpness}
	If $q$ is a square and $d \mid \sqrt{q} + 1$, then there exists a subset $A \subseteq \F_q$ such that $|A| = \sqrt{q}$ and $A$ does not contain any distinct elements whose difference is a $d$th power.
\end{proposition}

\begin{proof}
	Consider the generalized Paley graph $\bfP(q, d)$ with vertex set $V = \F_q$ and edges $\{a, b\} \in E$ if and only if $b - a = x^d$ for some $x \in \F_q$.
	Note that independent sets in $\bfP(q, d)$ correspond to subsets of $\F_q$ with no $d$th power differences.
	We therefore want to show that $\bfP(q, d)$ has an independent set of size $\sqrt{q}$.
	Under the assumption $d \mid \sqrt{q} + 1$, the chromatic number of $\bfP(q, d)$ is equal to $\sqrt{q}$ by \cite[Theorem 1]{bdr}.
	But this means that $\F_q$ can be partitioned into a collection of $\sqrt{q}$ independent sets, so there must be an independent set of size at least $\frac{q}{\sqrt{q}} = \sqrt{q}$.
\end{proof}


\section{Irrational equidistribution for polynomials over $\F_p$} \label{sec: irrational equidistribution}

As preparation for our next combinatorial application (Theorem \ref{thm: A,B Sarkozy}), we prove Theorem \ref{thm: irrational equidistribution}, reproduced below, which gives a simple characterization of when a polynomial is good for irrational equidistribution (see Definition \ref{defn: characters} for the definition).

\IrrationalEquidistribution*

To prove Theorem \ref{thm: irrational equidistribution}, we will combine Lemma \ref{lem: Euclidean} (proved in Section \ref{sec: additive core}) with a general Weyl-type equidistribution theorem from \cite{bl}.
Let us first introduce some notation.
Let $\F_p(t) = \left\{ f/g : f, g \in \F_p[t], g \ne 0 \right\}$ be the field of rational functions over $\F_p$.
We define an absolute value on $\F_p(t)$ by $|f/g| = p^{\deg{f} - \deg{g}}$ with the convention that $\deg{0} = -\infty$.
The completion of $\F_p(t)$ with respect to the metric induced by $|\cdot|$ is the field of formal Laurent series $\F_p((t^{-1})) = \left\{ \sum_{n=-\infty}^N c_n t^n : N \in \Z, c_n \in \F_p \right\}$.
We call an element $\alpha \in \F_p((t^{-1}))$ \emph{rational} if $\alpha \in \F_p(t)$ and \emph{irrational} otherwise.
Rational elements of $\F_p((t^{-1}))$ share many of the familiar properties of rational numbers:

\begin{proposition}
	Let $\alpha = \sum_{n=-\infty}^N c_n t^n \in \F_p((t^{-1}))$.
	The following are equivalent:
	\begin{enumerate}[(i)]
		\item	$\alpha \in \F_p(t)$;
		\item	the sequence of ``digits'' $(c_n)_{-\infty < n \le N}$ is eventually periodic: there exists $M \in \Z$ and $q \in \N$ such that $c_{n-q} = c_n$ for all $n \le M$;
		\item	the sequence $(f\alpha)_{f \in \F_p[t]}$ is periodic mod $\F_p[t]$: there exists $g \in \F_p[t] \setminus \{0\}$ such that for any $f, h \in \F_p[t]$, one has $(f + gh)\alpha - f\alpha \in \F_p[t]$;
		\item	the sequence $(f\alpha)_{f \in \F_p[t]}$ has finitely many elements mod $\F_p[t]$: there exists $k \in \N$ and elements $\beta_1, \dots, \beta_k \in \F_p((t^{-1}))$ such that for any $f \in \F_p[t]$, there exists $i \in \{1, \dots, k\}$ such that $f\alpha - \beta_i \in \F_p[t]$.
	\end{enumerate}
\end{proposition}

\begin{proof}
	(i) $\implies$ (iii).
	Write $\alpha = \frac{f}{g}$ with $f, g \in \F_p[t]$, $g \ne 0$.
	Then for any $h_1, h_2 \in \F_p[t]$, we have $(h_1 + gh_2)\alpha - h_1 \alpha = fh_2 \in \F_p[t]$.
	
	(iii) $\implies$ (iv).
	Let $g$ be as in (iii), and let $h_1, \dots, h_k$ be the finitely many elements $h_i \in \F_p[t]$ such that $|h_i| < |g|$.
	Put $\beta_i = h_i \alpha$.
	Let $f \in \F_p[t]$.
    The remainder from the division of $f$ by $g$ is an element of $\F_p[t]$ of size smaller than $g$, so it is equal to $h_i$ for some $i \in \{1, \dots, k\}$.
    Hence, $f\alpha - \beta_i = (f - h_i)\alpha \in \F_p[t]$, since $f - h_i$ is divisible by $g$.
	
	(iv) $\implies$ (ii).
	Note that $t^m \alpha = \sum_{n=-\infty}^{N+m} c_{n-m} t^n$.
	By (iv), the sequence $(t^m\alpha)_{m \in \N}$ has only finitely many elements mod $\F_p[t]$, so let $m_1 < m_2$ such that $t^{m_1} \alpha - t^{m_2} \alpha \in \F_p[t]$.
	Then comparing coefficients, we have $(c_{-(m_1+1)}, c_{-(m_1+2)}, \dots) = (c_{-(m_2+1)}, c_{-(m_2+2)}, \dots)$.
	Thus for $M = -(m_1 + 1)$ and $q = m_2 - m_1$, we have $c_{n-q} = c_n$ for all $n \le M$.
	
	(ii) $\implies$ (i).
	Let $M \in \Z$ and $q \in \N$ such that $c_{n-q} = c_n$ for $n \le M$.
	We can then write
	\begin{multline*}
		\alpha = \sum_{n=M+1}^N c_n t^n + \left( c_M \left( t^M + t^{M-q} + t^{M-2q} + \dots \right) + \dots + c_{M-q+1} \left( t^{M-q+1} + t^{M-2q+1} + t^{M-3q+1} + \dots \right) \right) \\
		 = \sum_{n=M+1}^N c_n t^n + \left( c_M t^M + \dots + c_{M-q+1} t^{M-q+1} \right) \frac{t^q}{t^q - 1} \in \F_p(t).
	\end{multline*}
\end{proof}

There is an isomorphism between the dual group $\hat{\F_p[t]}$ of additive characters on $\F_p[t]$ and the characteristic $p$ ``torus'' $\F_p((t^{-1}))/\F_p[t]$.
Indeed, every character on $\F_p[t]$ is of the form $f \mapsto e(\alpha f)$ for some $\alpha \in \F_p((t^{-1}))/\F_p[t]$, where $e \left( \sum_{n=-\infty}^N c_n \right) = \exp \left( \frac{2\pi i c_{-1}}{p} \right)$.
(Given any $p$th root of unity $\omega$, one can define $e_{\omega} \left( \sum_{n=-\infty}^N c_n \right) = \omega^{c_{-1}}$ and obtain in this way another isomorphism between $\hat{\F_p[t]}$ and $\F_p((t^{-1}))/\F_p$.
Changing the choice of $\omega$ does not impact the discussion below.)
A key property of this isomorphism for our purposes is that a character $\chi(f) = e(\alpha f)$ is rational (see Definition \ref{defn: characters}) if and only if $\alpha \in \F_p(t)$ is a rational element.

A function $a : \F_p[t] \to \F_p((t^{-1}))$ is \emph{uniformly distributed mod $\F_p[t]$} if for any continuous function $F : \F_p((t^{-1}))/\F_p[t]$, one has
\begin{equation*}
	\lim_{N \to \infty} \E_{f \in \M_N} F(a(f)) = \int_{\F_p((t^{-1}))/\F_p[t]} F~dm,
\end{equation*}
where $m$ is the Haar probability measure on $\F_p((t^{-1}))/\F_p[t]$.
Equivalently (by the Stone--Weierstrass theorem),
\begin{equation*}
	\lim_{N \to \infty} \E_{f \in \M_N} e(g a(f)) = 0
\end{equation*}
for every $g \in \F_p[t]$.
Thus, we see that a polynomial $P(x) \in \F_p[x]$ is good for irrational equidistribution if and only if $(P(f)\alpha)_{f \in \F_p[t]}$ is uniformly distributed mod $\F_p[t]$ for every irrational $\alpha \in \F_p((t^{-1})) \setminus \F_p(t)$.
(This is the source of our terminology ``good for irrational equidistribution.'')

The main result of \cite{bl} gives a description of the equidistributional behavior of polynomial sequences $P(x) \in \F_p((t^{-1}))[x]$.
In order to state and use this theorem, we first introduce some notation and a description of the behavior of additive polynomial sequences.

For any additive polynomial $\eta(x) \in \F_p((t^{-1}))[x]$, there is a closed subgroup $\mathcal{F}(\eta) \le \F_p((t^{-1}))/\F_p[t]$ such that the closure $\overline{\eta(\F_p[t])}$ of the image of $\eta$ in $\F_p((t^{-1}))/\F_p[t]$ takes the form $\mathcal{F}(\eta) + \eta(K)$ for some finite subset $K \subseteq \F_p[t]$ such that $\eta(K)$ is a finite subgroup of $\F_p((t^{-1}))/\F_p[t]$.
The group $\mathcal{F}(\eta)$ can be obtained explicitly as
\begin{equation*}
    \mathcal{F}(\eta) = \bigcap_{g \in \F_p[t] \setminus \{0\}} \overline{\eta(g\F_p[t])}.
\end{equation*}
In fact, there exists $g_0 \in \F_p[t] \setminus \{0\}$ such that if $g \in \F_p[t] \setminus \{0\}$ and $g_0 \mid g$, then
\begin{equation*}
    \mathcal{F}(\eta) = \overline{\eta(g\F_p[t])}.
\end{equation*}
(This description of $\mathcal{F}(\eta)$ can be gleaned from \cite[Section 7]{bl}.)

Given a polynomial $P(x) = \alpha_0 + \sum_{i=1}^n \eta_i(x^{r_i}) \in \F_p((t^{-1}))[x]$, we put $\mathcal{F}(P) = \sum_{i=1}^n \mathcal{F}(\eta_i)$.
The main theorem from \cite{bl} has the following consequence, as described in \cite[Section 3]{ab_corrigendum}:

\begin{theorem} \label{thm: Weyl}~
	Let $P(x) \in \F_p((t^{-1}))[x]$.
    If $\mathcal{F}(P) = \F_p((t^{-1}))/\F_p[t]$, then $(P(f))_{f \in \F_p[t]}$ is uniformly distributed mod $\F_p[t]$.
\end{theorem}

For $P(x) \in \F_p[t]$ and $\alpha \in \F_p((t^{-1}))$, the next lemma describes the group $\mathcal{F}(P\alpha)$ appearing in Theorem \ref{thm: Weyl} in terms of the additive core of $P$.

\begin{lemma} \label{lem: F-subtorus additive core}
    Let $P(x) \in \F_p[x]$, and let $\eta$ be its additive core.
    For any $\alpha \in \F_p((t^{-1}))$, $\mathcal{F}(P\alpha) = \mathcal{F}(\eta\alpha)$.
\end{lemma}

\begin{proof}
    Write $P(x) = a_0 + \sum_{i=1}^n \eta_i(x^{r_i})$ with $\eta_i$ additive and $x^{r_i}$ distinct and separable.
    The additive core $\eta$ satisfies $\eta(g\F_p[t]) = \sum_{i=1}^n \eta_i(g\F_p[t])$ for every $g \in \F_p[t] \setminus \{0\}$ (see Remark \ref{rem: additive core via other rings}).
    By definition,
    \begin{equation*}
        \mathcal{F}(P\alpha) = \sum_{i=1}^n \mathcal{F}(\eta_i\alpha).
    \end{equation*}
    For each $i \in \{1, \ldots, n\}$, let $g_i \in \F_p[t] \setminus \{0\}$ such that
    \begin{equation*}
        \mathcal{F}(\eta_i\alpha) = \overline{\eta_i(g_i\F_p[t])\alpha}.
    \end{equation*}
    Let $g_0 \in \F_p[t] \setminus \{0\}$ such that
    \begin{equation*}
        \mathcal{F}(\eta\alpha) = \overline{\eta(g_0\F_p[t])\alpha}.
    \end{equation*}
    Set $g = \prod_{i=0}^n g_i \in \F_p[t] \setminus \{0\}$.
    Then
    \begin{equation*}
        \mathcal{F}(P\alpha) = \sum_{i=1}^n \overline{\eta_i(g\F_p[t])\alpha} = \overline{\sum_{i=1}^n \eta_i(g\F_p[t])\alpha} = \overline{\eta(g\F_p[t])} = \mathcal{F}(\eta\alpha).
    \end{equation*}
\end{proof}

We can now prove Theorem \ref{thm: irrational equidistribution}.

\begin{proof}[Proof of Theorem \ref{thm: irrational equidistribution}]
    Let $P(x) \in \F_p[x]$, and let $\eta$ be its additive core. \\

    First suppose $\eta(x) = ax$ for some $a \in \F_p^{\times}$.
    Let $\alpha \in \F_p((t^{-1})) \setminus \F_p(t)$ be irrational.
    Then $a\alpha$ is also irrational, so $\mathcal{F}(\eta\alpha) = \F_p((t^{-1}))/\F_p[t]$ by \cite[Theorem 3.1]{bl}.
    Therefore, by Lemma \ref{lem: F-subtorus additive core}, $\mathcal{F}(P\alpha) = \F_p((t^{-1}))/\F_p[t]$, whence $(P(f)\alpha)_{f \in \F_p[t]}$ is uniformly distributed mod $\F_p[t]$ by Theorem \ref{thm: Weyl}.
    Thus, $P$ is good for irrational equidistribution. \\

    Now suppose $\eta(x) = \sum_{j=0}^m a_j x^{p^j}$ with $a_m \ne 0$ and $m \ge 1$.
    We will find an irrational element $\alpha \in \F_p((t^{-1})) \setminus \F_p(t)$ such that $\overline{\eta(\F_p[t])\alpha} \ne \F_p((t^{-1}))/\F_p[t]$.
    We then deduce that
    \begin{equation*}
        \overline{P(\F_p[t])\alpha} \subseteq P(0) \alpha + \overline{\eta(\F_p[t])\alpha} \ne \F_p((t^{-1}))/\F_p[t],
    \end{equation*}
    so $P$ is not good for irrational equidistribution.

    If $a_0 = 0$, then $\eta(\F_p[t]) \subseteq \{f^p : f \in \F_p[t]\}$, so if we pick $\alpha = \beta^p$ for some irrational $\beta$, then
    \begin{equation*}
        \overline{\eta(\F_p[t])\alpha} \subseteq \{x^p : x \in \F_p((t^{-1}))/\F_p[t]\},
    \end{equation*}
    which is a proper subgroup of $\F_p((t^{-1}))/\F_p[t]$ (see \cite[p. 931, Example 1]{bl}).

    Suppose $a_0 \ne 0$.
    Then we may write $\eta(x) = x Q(x)$ with $Q(x) = a_0 + \sum_{j=1}^m a_j x^{p^j-1}$.
	We claim that the set
	\begin{equation*}
		R = \left\{ k \in \N : Q~\text{has a root in}~\F_{p^k} \right\}
	\end{equation*}
	is infinite.
	This is equivalent to proving that $R' = \left\{ g \in \F_p[t] \setminus \{0\} : g~\text{is irreducible and}~Q~\text{has a root} \bmod{g} \right\}$ is infinite, since $\F_{p^k}$ is the quotient of $\F_p[t]$ by an irreducible polynomial of degree $k$.
    We carry out a variant of Euler's proof of the infinitude of primes.
	Given a finite collection of irreducible polynomials $g_1, \dots, g_r \in \F_p[t] \setminus \{0\}$, consider
	\begin{equation} \label{eq: Euclid trick}
		Q(g_1 \dots g_r x) = a_0 + g_1 \dots g_r x \sum_{j=1}^m{a_j \left( g_1 \dots, g_r x \right)^{p^i-2}}.
	\end{equation}
	Since $Q$ is a nonzero polynomial, there exists $f \in \F_p[t]$ such that $Q(g_1 \dots g_r f) \ne 0$.
	From the expression on the right hand side of \eqref{eq: Euclid trick}, we have $Q(g_1 \dots g_r f) \equiv a_0 \pmod{g_i}$ for each $i \in \{1, \dots, r\}$.
	In particular, $g_i \nmid Q(g_1 \dots g_r f)$, so factoring $Q(g_1 \dots g_r f)$ into irreducibles, we find an irreducible polynomial $g \in \F_p[t] \setminus \{g_1, \dots, g_r\}$ such that $Q(g_1 \dots g_r f) \equiv 0 \pmod{g}$.
	Hence, $R'$ is infinite as claimed.
	
	Now let $k \in R$, and let $x \in \F_{p^k}$ with $Q(x) = 0$.
	Then $\eta(x) = xQ(x) = 0$, but $x \ne 0$, since $Q(0) = a_0 \ne 0$.
	Hence, the kernel of $\eta$ as an endomorphism of $\F_{p^k}$ is nontrivial, so $\eta(\F_{p^k})$ is a proper subgroup of $(\F_{p^k}, +)$.
    Therefore, there exists $\xi \in \F_{p^k} \setminus \{0\}$ such that $e_{p^k}(\xi \eta(x)) = 1$ for every $x \in \F_{p^k}$.
	Taking an isomorphism $\F_{p^k} \cong \F_p[t]/g\F_p[t]$ for an irreducible polynomial $g \in \F_p[t]$ of degree $k$, we may lift $\chi(x) = e_{p^k}(\xi x)$ to a $g$-periodic character on $\F_p[t]$ corresponding to a rational point $\frac{\tilde{\xi}}{g}$ for some $\tilde{\xi} \in \F_p[t]$, $\deg{\tilde{\xi}} < k$.
	Thus, the set
	\begin{equation*}
		A = \left\{ \alpha \in \F_p((t^{-1}))/\F_p[t] : e(\eta(f) \alpha) = 1~\text{for every}~f \in \F_p[t] \right\}
	\end{equation*}
	is infinite, since it contains a point of the form $\frac{\tilde{\xi}}{g}$ for each $k \in R$.
	But $A$ is a closed subgroup of the compact group $\F_p((t^{-1}))/\F_p[t]$, so it is uncountable.\footnote{This is a basic fact about compact groups for which we unfortunately do not know of any good reference.
	One can easily deduce this fact from the existence of a Haar probability measure on compact groups, but more elementary arguments are also possible, one of which we sketch now.
	Suppose for contradiction that $A$ is countably infinite.
	Then $\bigcap_{x \in A} (A \setminus \{x\}) = \es$, so by the Baire category theorem, at least one of the sets $A \setminus \{x\}$ must not be dense.
	That is, $A$ has an isolated point.
	But $A$ is a topological group, so it follows that every point in $A$ is isolated.
	An infinite collection of isolated points is non-compact, so we have reached a contradiction.}
	In particular, $A$ contains an irrational element.

    Let $\alpha \in A$ be an irrational element.
    Then $\overline{\eta(\F_p[t])\alpha}$ is annihilated by the character $x \mapsto e(x)$, so $\overline{\eta(\F_p[t])\alpha}$ is a proper subgroup of $\F_p((t^{-1}))/\F_p[t]$.
    This completes the proof.
\end{proof}


\section{An asymmetric Furstenberg--S\'{a}rk\"{o}zy theorem} \label{sec: A,B Sarkozy}

With an understanding of irrational equidistribution at hand from Theorem \ref{thm: irrational equidistribution}, we can now prove Theorem \ref{thm: A,B Sarkozy}, dealing with an asymmetric form of the Furstenberg--S\'{a}rk\"{o}zy theorem.

\ABSarkozy*

\begin{proof}
	(i) $\implies$ (iv).
	Let $Q(x) \in \F_p[x]$ and $s \ge 0$ such that $Q$ is good for irrational equidistribution and $P(x) = Q(x^{p^s})$.
	By Theorem \ref{thm: irrational equidistribution}, it follows that the additive core $\eta$ of $P$ is of the form $\eta(x) = a x^{p^s}$ for some $a \in \F_p^{\times}$.
	In particular, $\eta(\F_{p^k}) = \F_{p^k}$ for every $k \in \N$.
	We then apply Corollary \ref{cor: projection} and the Cauchy--Schwarz inequality:
	\begin{align*}
		\Big| \left| \left\{ (x,y) \in \F_{p^k} \times \F_{p^k} : \right.\right. & \left.\left. x \in A~\text{and}~x + P(y) \in B \right\} \right| - |A| |B| \Big| \\
         & = p^{2k} \left| \E_{x \in \F_{p^k}} \ind_A(x) \left( \E_{y \in \F_{p^k}} \ind_B(x + P(y)) - \E_{z \in \F_{p^k}} \ind_B(z) \right) \right| \\
		 & \le p^{2k} \norm{L^2(\F_{p^k})}{\ind_A} (d-1) p^{-k/2} \norm{L^2(\F_{p^k})}{\ind_B} \\
		 & = (d-1) p^{k/2} \sqrt{|A| |B|}.
	\end{align*} \\
	
	(iv) $\implies$ (ii).
	Let $\delta > 0$.
	Let $C > 0$ be the implicit constant in (iv).
	Set $K_1 = \floor{\log_p(C^2 \delta^{-1})} + 1$ so that $p^{K_1} > C^2 \delta^{-1}$, and suppose $k \ge K_1$.
	Let $A, B \subseteq \F_{p^k}$ with $|A| |B| \ge \delta p^{2k} > C^2 p^k$.
	By (iv),
	\begin{equation*}
		\Big| \left| \left\{ (x,y) \in \F_{p^k} \times \F_{p^k} : x \in A~\text{and}~x + P(y) \in B \right\} \right| - |A| |B| \Big| \le C p^{k/2} \sqrt{|A| |B|}.
	\end{equation*}
	In particular,
	\begin{equation*}
		\left| \left\{ (x,y) \in \F_{p^k} \times \F_{p^k} : x \in A~\text{and}~x + P(y) \in B \right\} \right| \ge \sqrt{|A| |B|} \left( \sqrt{|A| |B|} - C p^{k/2} \right) > 0.
	\end{equation*}
	Let $(x,y) \in \F_{p^k} \times \F_{p^k}$ with $x \in A$ and $x + P(y) \in B$ and put $a = x$, $b = x + P(y)$.
	Then $a \in A$, $b \in B$, and $b - a = P(y)$. \\
	
	(ii) $\implies$ (i).
	We prove the contrapositive.
	Let $\eta$ be the additive core of $P$, and suppose $\eta(x) = \sum_{j=0}^m a_j x^{p^j}$ with at least two nonzero coefficients.
	Take $s = \min\{0 \le j \le m : a_j \ne 0\}$, and put $\eta'(x) = \sum_{j=0}^{m-s} a_{s+j} x^{p^j}$ so that $\eta(x) = \eta'(x^{p^s})$.
	As in the proof of Theorem \ref{thm: irrational equidistribution}, the set
	\begin{equation*}
		R = \left\{ k \in \N : \frac{\eta'(x)}{x}~\text{has a root in}~\F_{p^k} \right\}
	\end{equation*}
	is infinite.
	Let $k \in R$, and let $H_k = \eta(\F_{p^k})$.
	Note that $H_k$ is a proper subgroup of $\F_{p^k}$, since $H_k = \eta'(\F_{p^k})$ and $\eta'$ has a nonzero root.
	Taking $A = H_k$ and $B = H_k + c$ a nontrivial coset, we have $b - a \in H_k + c$, while $P(x) \in H_k$ for every $a \in A, b \in B, x \in \F_{p^k}$.
	Moreover, $|A| \cdot |B| = |H_k|^2 \ge \left( \frac{p^k}{d} \right)^2$, so condition (ii) fails for $\delta = d^{-2}$. \\
	
	(ii) $\implies$ (iii).
	Let $\delta > 0$.
	Let $k \ge K_1(P, \delta)$, and suppose $A, B \subseteq \F_{p^k}$ with $|A| \cdot |B| \ge \delta p^{2k}$.
	Fix $c \in \F_{p^k}$.
	By the definition of $K_1$, there exist $x \in \F_{p^k}$, $a \in A$, and $b \in (c - B)$ such that $b - a = P(x)$.
	Writing $b = c - b'$ with $b' \in B$, we have $c = a + b' + P(x) \in A + B + S$.
	Since $c$ was arbitrary, this prove $A + B + S = \F_{p^k}$. \\
	
	(iii) $\implies$ (ii).
	Let $\delta > 0$.
	Let $k \ge K_2(P, \delta)$, and suppose $A, B \subseteq \F_{p^k}$ with $|A| \cdot |B| \ge \delta p^{2k}$.
	By the definition of $K_2$, we have $A + (-B) + S = \F_{p^k}$.
	In particular, $0 \in A + (-B) + S$, so there exist $a \in A$, $b \in B$, and $x \in \F_{p^k}$ such that $a - b + P(x) = 0$.
	That is, $b - a = P(x)$.
\end{proof}


\section{Partition regularity of polynomial equations over finite fields} \label{sec: partition regularity}

In this section, we obtain applications of Theorem \ref{thm: exponential sum} to partition regularity of polynomial equations over finite fields.
We are interested in problems of the following kind.
Given polynomials $P_1, P_2, P_3 \in \F_p[x]$ and a finite coloring of the field $\F_{p^k}$ (here, the number of colors should be thought of as fixed and the parameter $k$ very large), how many solutions $(x,y,z) \in \F_{p^k}$ of the equation $P_1(x) + P_2(y) + P_3(z) = 0$ are monochromatic?

As a starting point, we must first address the problem of counting the total number of solutions of equations of the form $P_1(x) + P_2(y) + P_3(z) = 0$.
When the equation defines a geometrically irreducible variety\footnote{A system of polynomial equations $P_1(x_1, \dots, x_d) = c_1, \dots, P_k(x_1, \dots, x_d) = c_k$ with $P_1, \dots, P_k \in \F_p[x_1,\dots,x_d]$ defines a \emph{geometrically irreducible variety} if the set of solutions $V \subseteq \overline{\F}_p^d$ over the algebraic closure $\overline{\F}_p$ cannot be written as a union of two sets $V_1$ and $V_2$ that are themselves sets of solutions of systems of polynomial equations. For a single equation $P_1(x) + P_2(y) + P_3(z) = 0$, this corresponds to the polynomial $P(x,y,z) = P_1(x) + P_2(y) + P_3(z)$ being an irreducible polynomial in $\overline{\F}_p[x,y,z]$.}, the work of Lang and Weil \cite{lw} provides a satisfactory answer: the number of solutions is approximately $p^{2k}$, with an error of size $O(p^{3k/2})$.\footnote{Lang and Weil are in fact able to provide strong estimates on the number of solutions of systems of polynomial equations of a much more general form, as long as the system defines a geometrically irreducible variety.}
The method of Lang and Weil uses induction on the dimension of the variety, with the Weil bound as the base case and an estimate on how many slices of the variety by hyperplanes may become reducible in order to carry out the induction step.
In order to count solutions of equations of the form $P_1(x) + P_2(y) + P_3(z) = 0$ without any irreducibility assumption, we take a slightly different approach.
Because of the special form of the equation, we write the number of solutions as a double sum
\begin{equation*}
    \sum_{a,z \in \F_{p^k}} f_1(a) f_2(- a - P_3(z)),
\end{equation*}
where $f_1(a)$ is the number of solutions of the equation $P_1(x) = a$ and $f_2(b)$ is the number of solutions of the equation $P_2(y) = b$.
We can then estimate the sum using Corollary \ref{cor: projection}.
(We should note that, similarly to Lang and Weil, the quantitative strength provided by our method relies on the Weil bound.)

Let us make a few basic observations about the equation $P_1(x) + P_2(y) + P_3(z) = 0$.
By collecting the constant terms together, we may assume $P_1(0) = P_2(0) = P_3(0) = 0$ and instead solve the equation
\begin{equation} \label{eq: polynomial equation}
	P_1(x) + P_2(y) + P_3(z) = c
\end{equation}
for some constant $c$.
One can give an algebraic criterion that $c$ must satisfy in order for this equation to be solvable over $\F_{p^k}$, which we now describe.
Let $H_{k,i}$ be the additive subgroup of $(\F_{p^k}, +)$ generated by $\{P_i(x) : x \in \F_{p^k}\}$, and let $H_k$ be the subgroup $H_k = H_{1,k} + H_{2,k} + H_{3,k}$.
We may compute the group $H_k$ explicitly using Lemma \ref{lem: Euclidean}.
First, $H_{k,i} = \eta_i(\F_{p^k})$, where $\eta_i$ is the additive core of $P_i$.
Next, we let $\eta$ be the additive polynomial produced via Lemma \ref{lem: Euclidean} from the additive polynomials $\eta_1, \eta_2, \eta_3$.
Then $H_k = \eta(\F_{p^k})$ for every $k \in \N$.
Clearly, for any $x, y, z \in \F_{p^k}$, one has $P_1(x) + P_2(y) + P_3(z) \in H_k$.
As the following proposition shows, the set of values $c \in \F_{p^k}$ for which \eqref{eq: polynomial equation} has solutions is exactly the subgroup $H_k$, provided that $k$ is sufficiently large (depending on the polynomials $P_1, P_2, P_3$).
Moreover, the number of solutions of \eqref{eq: polynomial equation} is roughly the same for every value of $c \in H_k$.

\begin{proposition} \label{prop: solutions sum of three polys}
	Let $P_1, P_2, P_3 \in \F_p[x]$ be nonconstant polynomials with $P_i(0) = 0$.
	For each $i \in \{1,2,3\}$, let $H_{k,i} = \left\langle P_i(x) : x \in \F_{p^k} \right\rangle \le (\F_{p^k}, +)$, and let $H_k = H_{k,1} + H_{k,2} + H_{k,3}$.
	Then for any $k \in \N$ and any $c \in H_k$,
	\begin{equation} \label{eq: number of solutions}
		\left| \left\{ (x,y,z) \in \F_{p^k}^3 : P_1(x) + P_2(y) + P_3(z) = c \right\} \right|
		  = \frac{p^{3k}}{|H_k|} + O \left( p^{3k/2} \right),
	\end{equation}
	In particular, if $k$ is sufficiently large, then \eqref{eq: polynomial equation} has a solution over $\F_{p^k}$ if and only if $c \in H_k$.
\end{proposition}

\begin{remark}
    The subgroup $H_k$ appearing in Proposition \ref{prop: solutions sum of three polys} satisfies the bound
    \begin{equation*}
        \frac{p^k}{d} \le |H_k| \le p^k,
    \end{equation*}
    where $d = \min\{\deg{P_1}, \deg{P_2}, \deg{P_3}\}$.
    Indeed, $H_{k_i} = \eta_i(\F_{p^k})$ for an additive polynomial $\eta_i$ (the additive core of $P_i$) with degree at most $\deg{P_i}$, and
    \begin{equation*}
        |H_k| \ge |H_{k,i}| = \frac{p^k}{\ker \eta_i}.
    \end{equation*}
    Therefore, it follows from \eqref{eq: number of solutions} that if $c \in H_k$, then the number of solutions $(x,y,z) \in \F_{p^k}$ of the equation $P_1(x) + P_2(y) + P_3(z) = c$ is of order $p^{2k}$.
\end{remark}

\begin{proof}
	For any $k \in \N$, any $c \in \F_{p^k}$, and any functions $f_1, \dots, f_r : \F_{p^k} \to \F_{p^k}$, let $N(k,c; f_1, \dots, f_r)$ denote the number of solutions $(x_1, \dots, x_r) \in \F_{p^k}^r$ to the equation $\sum_{i=1}^r{f_i(x_i)} = c$.
	Our goal is to show
	\begin{equation*}
		N(k, c; P_1, P_2, P_3) = \frac{p^{3k}}{|H_k|} + O \left( p^{3k/2} \right)
	\end{equation*}
	for $c \in H_k$.
	
	For each $i \in \{1, 2, 3\}$, let $\eta_i$ be the additive core of $P_i$ so that $H_{k,i} = \eta_i(\F_{p^k})$, and let $d_i = \deg{P_i}$. \\
	
	\underline{Claim}: $N(k, c; P_1, P_2, P_3) = N(k, c; P_1, P_2, \eta_3) + O \left( p^{3k/2} \right)$.
	 
	 Fix $k \in \N$.
	 Let $f_i(x) = N(k, x; P_i)$.
	 Then
	 \begin{equation*}
	 	N(k, c; P_1, P_2, P_3) = \sum_{x, y \in \F_{p^k}}{f_1(x) f_2(c - x - P_3(y))} = p^{2k} \E_{x,y \in \F_{p^k}}{f_1(x) f_2(c - x - P_3(y))}.
	 \end{equation*}
	 Similarly,
	 \begin{equation*}
	 	N(k, c; P_1, P_2, \eta_3) = p^{2k} \E_{x,y \in \F_{p^k}}{f_1(x) f_2(c - x - \eta_3(y))} = p^{2k} \E_{x \in \F_{p^k}}{\E_{z \in H_{k,3}}{f_1(x) f_2(c - x - z)}}.
	 \end{equation*}
	 Hence, by the Cauchy--Schwarz inequality,
	 \begin{multline*}
	 	\left| N(k, c; P_1, P_2, P_3) - N(k, c; P_1, P_2, \eta_3) \right| \\
		 \le p^{2k} \norm{L^2(\F_{p^k})}{f_1} \norm{L^2(\F_{p^k})}{\E_{y \in \F_{p^k}}{f_2(c - x - P_3(y))} - \E_{z \in H_{k,3}}{f_2(c - x - z)}}.
	 \end{multline*}
	Now, by Corollary \ref{cor: projection},
	\begin{equation*}
		\norm{L^2(\F_{p^k})}{\E_{y \in \F_{p^k}}{f_2(c - x - P_3(y))} - \E_{z \in H_{k,3}}{f_2(c - x - z)}} \le (d_3-1) p^{-k/2} \norm{L^2(\F_q)}{f_2}.
	\end{equation*}
	Finally, for each $x \in \F_{p^k}$, the polynomial equation $P_i(u) = x$ has at most $d_i$ solutions $u \in \F_{p^k}$, so $\norm{L^2(\F_{p^k})}{f_i} \le \norm{L^{\infty}(\F_{p^k})}{f_i} \le d_i$.
	Putting everything together,
	\begin{equation*}
		N(k, c; P_1, P_2, P_3) = N(k, c; P_1, P_2, \eta_3) + O \left( p^{3k/2} \right)
	\end{equation*}
	as claimed. \\
	
	Applying the claim also to $P_1$ and $P_2$, we obtain the estimate
	\begin{equation*}
		N(k, c; P_1, P_2, P_3) = N(k, c; \eta_1, \eta_2, \eta_3) + O \left( p^{3k/2} \right).
	\end{equation*}
	Let $\eta : \F_{p^k}^3 \to \F_{p^k}$, $\eta(x,y,z) = \eta_1(x) + \eta_2(y) + \eta_3(z)$.
	Then $\eta$ is a group homomorphism with image $H_k$.
	Therefore, $N(k, c; \eta_1, \eta_2, \eta_3) = \left| \eta^{-1}(\{c\}) \right|$ is constant in $c \in H_k$, so
	\begin{equation*}
		N(k, c; \eta_1, \eta_2, \eta_3) = \frac{\left| \F_{p^k}^3 \right|}{|H_k|} = \frac{p^{3k}}{|H_k|}.
	\end{equation*}
\end{proof}

As discussed above, the class of polynomials handled by Proposition \ref{prop: solutions sum of three polys} is very restricted in comparison to the results of \cite{lw}.
However, the elementary method of proof allows us to avoid any irreducibility assumption and is more flexible for combinatorial enhancements, such as the following Ramsey-theoretic result, restated from the introduction:

\PartitionRegularity*

\begin{remark}
	The assumption that $Q$ is FF$_p$-intersective is necessary in Theorem \ref{thm: P-P=Q}.
	If $Q$ is not FF$_p$-intersective, then there is a sequence $k_n \to \infty$ such that $H_n = \left\langle Q(x) - Q(0) : x \in \F_{p^{k_n}} \right\rangle \le (\F_{p^{k_n}}, +)$ does not contain $Q(0)$ (see property (iv) in the proof of Theorem \ref{thm: FF intersective}).
	We may color $\F_{p^{k_n}}$ by cosets of $H_n$.
	The number of colors is equal to the index of $H_n$, which is bounded by $\deg{Q}$, so by refining the sequence $(k_n)_{n \in \N}$, we may assume the number of colors is a constant $r$.
	But for this sequence of $r$-colorings, the equation $x - y = Q(z)$ does not have any monochromatic solutions.
\end{remark}

Our proof of Theorem \ref{thm: P-P=Q} combines Theorem \ref{thm: exponential sum} with tools from the theory of Loeb measures on ultraproduct spaces and a technique from \cite{b_density-schur, ert-update} previously used to establish partition regularity of the equation $x - y = z^2$ over the integers.
We will need the following generalization of Corollary \ref{cor: projection}:

\begin{proposition} \label{prop: diagonal action}
	Let $P(x) \in \F_p[x]$ be a nonconstant polynomial of degree $d$, and let $\eta$ be its additive core.
	Then for any $k \in \N$, any $m \in \N$, and any $f : \F_{p^k}^m \to \C$,
	\begin{multline*}
		\norm{L^2(\F_{p^k}^m)}{\E_{y \in \F_{p^k}}{f(x_1 + P(y), \dots, x_m + P(y)))} - \E_{z \in H_k}{f(x_1 + a_0 + z, \dots, x_m + a_0 + z)}} \\
         \le (d-1) p^{-k/2} \norm{L^2(\F_{p^k}^m)}{f},
	\end{multline*}
	where $H_k = \eta(\F_{p^k})$ and $a_0 = P(0)$.
\end{proposition}

\begin{proof}
	Define $F : \F_{p^k}^m \to \C$ by
	\begin{equation*}
		F(\bm{x}) = \E_{y \in \F_{p^k}}{f(x_1 + P(y), \dots, x_m + P(y)))} - \E_{z \in H_k}{f(x_1 + a_0 + z, \dots, x_m + a_0 + z)}.
	\end{equation*}
	Then for $\bm{\xi} = (\xi_1, \dots, \xi_m) \in \F_{p^k}^m$, we have
	\begin{equation*}
		\hat{F}(\bm{\xi}) = \hat{f}(\bm{\xi}) \left( \E_{y \in \F_{p^k}} e \left( \sum_{i=1}^m \xi_i P(y) \right) - e \left( \sum_{i=1}^m \xi_i a_0 \right) \ind_{H_k^{\perp}} \left( \sum_{i=1}^m \xi_i \right) \right).
	\end{equation*}
	Theorem \ref{thm: exponential sum 2} gives the bound
	\begin{equation*}
		\left| \E_{y \in \F_{p^k}} e \left( \sum_{i=1}^m \xi_i P(y) \right) - e \left( \sum_{i=1}^m \xi_i a_0 \right) \ind_{H_k^{\perp}} \left( \sum_{i=1}^m \xi_i \right) \right| \le (d-1) p^{-k/2}.
	\end{equation*}
	Therefore, by Parseval's identity,
	\begin{equation*}
		\norm{L^2(\F_{p^k}^m)}{F} \le (d-1) p^{k/2} \left( \sum_{\bm{\xi} \in \F_{p^k}^m} \left| \hat{f} \left( \bm{\xi} \right) \right|^2 \right)^{1/2} = (d-1) p^{-k/2} \norm{L^2(\F_{p^k}^m)}{f}.
	\end{equation*}
\end{proof}

The relevant constructions for employing measure theory on ultraproducts are summarized as follows:

\begin{definition} ~
	\begin{itemize}
		\item	An \emph{ultrafilter} on $\N$ is a collection $\mU \subseteq \mathcal{P}(\N)$ of nonempty subsets of $\N$ such that:
			\begin{itemize}
				\item	if $A, B \in \mU$, then $A \cap B \in \mU$;
				\item	for any $A \subseteq \N$, either $A \in \mU$ or $\N \setminus A \in \mU$.
			\end{itemize}
			The ultrafilter $\mU$ is \emph{principal} if $\mU = \{A \subseteq \N : n \in A\}$ for some $n \in \N$ and \emph{non-principal} otherwise.
			The space of ultrafilters is denoted $\beta \N$.
		\item	Given $\mU \in \beta \N$ and a family of sets $(X_n)_{n \in \N}$, the \emph{ultraproduct} is the set
			\begin{equation*}
				\prod_{n \to \mU}{X_n} = \left( \prod_{n \in \N}{X_n} \right)/\equiv_{\mU},
			\end{equation*}
			where $\equiv_{\mU}$ is the equivalence relation defined by $(x_n)_{n \in \N} \equiv_{\mU} (y_n)_{n \in \N}$ if and only if $\{n \in \N : x_n = y_n\} \in \mU$.
		\item	Given $\mU \in \beta \N$ and a sequence $(x_n)_{n \in \N}$ taking values in a compact Hausdorff space $X$, the \emph{limit of $(x_n)_{n \in \N}$ along $\mU$} is defined to be the unique point\footnote{Such a point $x$ exists by compactness and is unique by the Hausdorff property.} $x \in X$ such that for any neighborhood $U$ of $x$, one has $\{n \in \N : x_n \in U\} \in \mU$.
			The limit of $(x_n)_{n \in \N}$ along $\mU$ is denoted by $\lim_{n \to \mU}{x_n}$.
		\item	Let $\mU \in \beta \N$, and let $(X_n, \mX_n, \mu_n)_{n \in \N}$ be a family of probability spaces.
			Let $X = \prod_{n \to \mU}{X_n}$.
			\begin{itemize}
				\item	An \emph{internal set} is a set of the form $\prod_{n \to \mU}{A_n}$ with $A_n \in \mX_n$.
				\item	The \emph{Loeb $\sigma$-algebra} $\mX$ is the $\sigma$-algebra on $X$ generated by the algebra of internal sets.
				\item	The \emph{Loeb measure} $\mu$ is the unique probability measure on $\mX$ with the property
					\begin{align*}
						\mu(A) = \lim_{n \to \mU}{\mu_n(A_n)}
					\end{align*}
					for any internal set $A = \prod_{n \to \mU}{A_n}$.
			\end{itemize}
	\end{itemize}
\end{definition}

The main property of the Loeb measure that we will use is the following version of Fubini's theorem:

\begin{proposition}[cf. \cite{keisler}, Theorem 1.12] \label{prop: Loeb Fubini}
	Let $(X_n)_{n \in \N}$ and $(Y_n)_{n \in \N}$ be sequences of finite sets.
	Let $\mU$ be a non-principal ultrafilter.
	Let $X = \prod_{n \to \mU}{X_n}$ and $Y = \prod_{n \to \mU}{Y_n}$.
	Let $f : X \times Y \to \C$ be a bounded Loeb-measurable function.
	Then
	\begin{enumerate}[(1)]
		\item	for any $x \in X$, the function $y \mapsto f(x,y)$ is Loeb-measurable on $Y$;
		\item	the function $x \mapsto \int_Y{f(x,y)~d\mu_Y(y)}$ is Loeb-measurable on $X$; and
		\item	\begin{equation*}
				\int_{X \times Y}{f~d\mu_{X \times Y}}
				 = \int_X{\left( \int_Y{f(x,y)~d\mu_Y(y)} \right)~d\mu_X(x)}.
			\end{equation*}
	\end{enumerate}
\end{proposition}

\begin{remark}
	Proposition \ref{prop: Loeb Fubini} does not follow from the standard version of Fubini's theorem.
	The subtlety lies in the structure of the Loeb $\sigma$-algebra on the product space $X \times Y$: there are internal subsets of $X \times Y$ that cannot be approximated by Boolean combinations of Cartesian products of internal subsets of $X$ and $Y$ (on the finitary level, this corresponds to approximating subsets of $X_n \times Y_n$ by products of boundedly many subsets of $X_n$ and $Y_n$).
	Therefore, the function $f$ need not be measurable with respect to the product of the Loeb $\sigma$-algebras on $X$ and $Y$.
	Nevertheless, Proposition \ref{prop: Loeb Fubini} shows that $\mu_{X \times Y}$ shares important features with the product measure $\mu_X \times \mu_Y$.
\end{remark}

\begin{proof}[Proof of Theorem \ref{thm: P-P=Q}]
	Let $r \in \N$.
	Suppose for contradiction that there are $r$-colorings $\F_{p^{k_n}} = \bigcup_{i=1}^r{C_{n,i}}$ with $k_n \to \infty$ such that $|M_n| = o_{n \to \infty} \left( p^{2k_n} \right)$, where
	\begin{equation*}
		M_n = \left\{ (x,y,z) \in \F_{p^{k_n}}^3 : P(x) - P(y) = Q(z)~\text{and}
		 ~\{x,y,z\} \subseteq C_{n,i}~\text{for some}~i \in \{1, \dots, r\} \right\}
	\end{equation*}
	is the collection of monochromatic solutions to the equation $P(x) - P(y) = Q(z)$.
	
	Now we define a limit object associated with this sequence of colorings.
	Fix a non-principal ultrafilter $\mU$ on $\N$.
	Let $\F_{\infty}$ be the pseudo-finite field $\F_{\infty} = \prod_{n \to \mU}{\F_{p^{k_n}}}$, let $C_i = \prod_{n \to \mU}{C_{n,i}} \subseteq \F_{\infty}$, and let $M = \prod_{n \to \mU}{M_n}$.
	Denote by $\mu$ the Loeb measure on $\F_{\infty}$ obtained by equipping $\F_{p^{k_n}}$ with the normalized counting measure.
	For any $s \in \N$, we denote the Loeb measure on $\F_{\infty}^s$ by $\mu^s$ (not to be confused with the product measure $\mu \times \dots \times \mu$ on $\F_{\infty}^s$).
	Let $V_n = \left\{ (x,y,z) \in \F_{p^{k_n}}^3 : P(x) - P(y) = Q(z) \right\}$ and $V = \prod_{n \to \mU}{V_n}$.
	Finally, let $\mu_V$ be the Loeb measure on $V$ obtained from the normalized counting measures on $V_n$. \\
	
	\underline{Claim 1}: $\F_{\infty} = \bigcup_{i=1}^r{C_i}$.
	
	Let $x = (x_n)_{n \in \N} \in \F_{\infty}$.
	For $i \in \{1, \dots, r\}$, let $I_i = \left\{ n \in \N : x_n \in C_{n,i} \right\}$.
	Then $\N = \bigcup_{i=1}^r{I_i}$, so $I_{i_0} \in \mU$ for some $i_0 \in \{1, \dots, r\}$, since $\mU$ is an ultrafilter.
	By the definition of the sets $C_1, \dots, C_r$, it follows that $x \in C_{i_0}$.
	This proves the claim. \\
	
	Arguing as in the proof of Claim 1 above, one can check that $M$ is the set of monochromatic solutions $(x,y,z) \in \F_{\infty}^3$ to the equation $P(x) - P(y) = Q(z)$ with respect to the coloring $\F_{\infty} = \bigcup_{i=1}^r{C_i}$. \\
	
	\underline{Claim 2}: $\mu_V(M) = 0$.
	
	We have constructed $M$ as an internal set, so by the definition of the Loeb measure,
	\begin{equation*}
		\mu_V(M) = \lim_{n \to \mU}{\frac{|M_n|}{|V_n|}}.
	\end{equation*}
	Now, by Proposition \ref{prop: solutions sum of three polys}, $|V_n| = \frac{p^{3k_n}}{|H_{k_n}|} + O \left( p^{3k_n/2} \right)$, where $H_{k_n}$ is the subgroup generated by $\{P(x) - P(y) - Q(z) : x, y, z \in \F_{p^{k_n}}\}$.
	Noting that $\frac{p^{k_n}}{\min\{\deg{P}, \deg{Q}\}} \le |H_{k_n}| \le p^{k_n}$, we have
	\begin{equation} \label{eq: asymptotics of V_n}
		1 \le \liminf_{n \to \infty} \frac{|V_n|}{p^{2k_n}} \le \limsup_{n \to \infty} \frac{|V_n|}{p^{2k_n}} < \infty.
	\end{equation}
	By assumption, $|M_n| = o \left( p^{2k_n} \right)$.
	Hence, $\frac{|M_n|}{|V_n|} = o(1)$, so $\mu_V(M) = 0$, since $\mU$ is non-principal. \\
	
	Let $A_i = P(C_i) = \prod_{n \to \mU}{P(C_{n,i})}$.
	Note that
	\begin{equation*}
		\frac{1}{d} \mu(C_i) \le \mu(A_i) \le \mu(C_i),
	\end{equation*}
	where $d = \deg{P}$.
	In particular, $\mu(A_i) = 0$ if and only if $\mu(C_i) = 0$.
	
	Without loss of generality, we may assume that $\mu(C_i) > 0$ for $1 \le i \le s$ and $\mu(C_i) = 0$ for $s+1 \le i \le r$, for some $s \in \{1, \dots, r\}$.
	Let $A = A_1 \times \dots \times A_s \subseteq \F_{\infty}^s$.
	Note that $\mu^s(A) = \prod_{i=1}^s{\mu(A_i)} > 0$ by Proposition \ref{prop: Loeb Fubini}.
	Let $T_z : \F_{\infty}^s \to \F_{\infty}^s$ be the map $T_z \bm{x} = (x_1 + z, \dots, x_s + z)$ for $z \in \F_{\infty}$, $\bm{x} = (x_1, \dots, x_s) \in \F_{\infty}^s$.
	For each $n \in \N$ and $i \in \{1, \dots, r\}$, let $A_{n,i} = P(C_{n,i})$, and let $A^{(n)} = A_{n,1} \times \dots \times A_{n,s} \in \F_{p^{k_n}}^s$.
	Also let $T^{(n)}_z : \F_{p^{k_n}}^s \to \F_{p^{k_n}}^s$ be the map $T^{(n)}_z \bm{x} = (x_1 + z, \dots, x_s + z)$ for $z \in \F_{p^{k_n}}$ and $\bm{x} = (x_1, \dots, x_s) \in \F_{p^{k_n}}^s$.
	Now, since $Q$ is FF$_p$-intersective, we have
	\begin{equation} \label{eq: diagonal action asymptotics}
		\E_{\bm{x} \in \F_{p^{k_n}}^s} \E_{z \in \F_{p^{k_n}}} \ind_{A^{(n)}}(\bm{x}) \ind_{A^{(n)}} \left( T^{(n)}_{Q(z)} \bm{x} \right) = \E_{\bm{x} \in \F_{p^{k_n}}^s} \E_{y \in H_{k_n}} \ind_{A^{(n)}}(\bm{x}) \ind_{A^{(n)}} \left( T^{(n)}_y \bm{x} \right) + o_{n \to \infty}(1)
	\end{equation}
	by Proposition \ref{prop: diagonal action} and the Cauchy--Schwarz inequality.
	Hence,
	\begin{align*}
		\int_{\F_{\infty}} \mu^s \left( A \cap T_{Q(z)} A \right)~d\mu(z)
		 & \stackrel{(1)}{=} \int_{\F_{\infty}^{s+1}} \ind_A(\bm{x}) \ind_A(T_{Q(z)} \bm{x})~d\mu^{s+1}(\bm{x},z) \\
		 & \stackrel{(2)}{=} \lim_{n \to \mU} \E_{(\bm{x},z) \in \F_{p^{k_n}}^{s+1}} \ind_{A^{(n)}}(\bm{x}) \ind_{A^{(n)}} \left( T^{(n)}_{Q(z)} \bm{x} \right) \\
		 & \stackrel{(3)}{=} \lim_{n \to \mU} \E_{\bm{x} \in \F_{p^{k_n}}^s} \E_{y \in H_{k_n}} \ind_{A^{(n)}}(\bm{x}) \ind_{A^{(n)}} \left( T^{(n)}_y \bm{x} \right) \\
		 & \stackrel{(4)}{\ge} \lim_{n \to \mU}{\left( \frac{|A^{(n)}|}{p^{sk_n}} \right)^2} \\
		 & \stackrel{(5)}{=} \mu^s(A)^2 > 0.
	\end{align*}
	The steps are justified as follows.
	Step (1) is a direct application of Proposition \ref{prop: Loeb Fubini}.
	The equality (2) comes from the definition of the Loeb measure $\mu^{s+1}$.
	In step (3), we have taken the limit of both sides of \eqref{eq: diagonal action asymptotics} along $\mU$.
	The inequality (4) holds for each $n \in \N$ by the Cauchy--Schwarz inequality:
	\begin{align*}
		\E_{\bm{x} \in \F_{p^{k_n}}^s} \E_{y \in H_{k_n}} \ind_{A^{(n)}}(\bm{x}) \ind_{A^{(n)}} \left( T^{(n)}_y \bm{x} \right)
		 & = \innprod{\ind_{A^{(n)}}}{\E_{y \in H_{k_n}} T^{(n)}_y \ind_{A^{(n)}}} \\
		 & = \norm{L^2 \left( \F_{p^{k_n}}^s \right)}{\E_{y \in H_{k_n}} T^{(n)}_y \ind_{A^{(n)}}}^2 \\
		 & \ge \left( \frac{\left| A^{(n)} \right|}{p^{sk_n}} \right)^2.
	\end{align*}
	Finally, (5) follows from the definition of the Loeb measure $\mu^s$.
	
	Thus,
	\begin{equation*}
		\mu \left( \left\{ z \in \F_{\infty} : \mu^s \left( A \cap T_{Q(z)} A \right) > 0 \right\} \right) > 0.
	\end{equation*}
	Let $G = \left\{ z \in \bigcup_{i=1}^s{C_i} : \mu^s \left( A \cap T_{Q(z)} A \right) > 0 \right\}$.
	Since the set $\bigcup_{i=s+1}^r{C_i}$ has Loeb measure zero, $\mu(G) > 0$.
	
	For $z \in G$, let $i(z) \in \{1, \dots, s\}$ such that $z \in C_{i(z)}$.
	Noting that
	\begin{equation*}
		\mu^s \left( A \cap T_{Q(z)}A \right) = \prod_{i=1}^s{\mu \left( A_i \cap (A_i + Q(z)) \right)}
	\end{equation*}
	by Proposition \ref{prop: Loeb Fubini}, it follows that
	\begin{equation*}
		\mu \left( A_{i(z)} \cap \left( A_{i(z)} + Q(z) \right) \right) > 0.
	\end{equation*}
	The set $A_{i(z)}$ lies in the image of $P$ by definition, so taking the inverse image under $P$,
	\begin{equation*}
		\mu \left( C_{i(z)} \cap P^{-1} \left( A_{i(z)} + Q(z) \right) \right) > 0.
	\end{equation*}
	Therefore, letting $\alpha = \lim_{n \to \mU} \frac{|V_n|}{p^{2k_n}} \in [1, \infty)$ (see \eqref{eq: asymptotics of V_n}), we have
	\begin{align*}
		\mu_V(M) & = \lim_{n \to \mU} \frac{|M_n|}{|V_n|} \\
		 & = \alpha^{-1} \lim_{n \to \mU} \frac{|M_n|}{p^{2k_n}} \\
		 & = \alpha^{-1} \lim_{n \to \mU} \frac{1}{p^{2k_n}} \sum_{i=1}^r \sum_{z \in \F_{p^{k_n}}} \ind_{C_{n,i}}(z) \left| \left\{ (x,y) \in C_{n,i}^2 : P(x) - P(y) = Q(z) \right\} \right| \\
		 & \ge \alpha^{-1} \sum_{i=1}^r \lim_{n \to \mU} \E_{z \in \F_{p^{k_n}}} \ind_{C_{n,i}}(z) \frac{\left| C_{n,i} \cap P^{-1} \left( A_{n,i} + Q(z) \right) \right|}{p^{k_n}} \\
		 & = \alpha^{-1} \sum_{i=1}^r \int_{\F_{\infty}} \ind_{C_i}(z) \mu \left( C_i \cap P^{-1} \left( A_i + Q(z) \right) \right) \\
		 & \ge \alpha^{-1} \int_G{\mu \left( C_{i(z)} \cap P^{-1} \left( A_{i(z)} + Q(z) \right) \right)~d\mu(z)} \\
		 & > 0.
	\end{align*}
	
	This final inequality contradicts Claim 2, so we are done.
\end{proof}

An important feature of the proof of Theorem \ref{thm: P-P=Q} is the following.
Taking the ultraproduct of a sequence of finite fields $\F_q$ with characteristic growing to infinity, the same method shows that for any nonconstant polynomials $P(x), Q(x) \in \Z[x]$, the equation $P(x) - P(y) = Q(z)$ is partition regular over all fields of sufficiently high characteristic.
This follows by noting that $P$ and $Q$ will be nonconstant and separable (hence good for irrational equidistribution; see Example \ref{eg: irrational equidistribution}(1) above) once the characteristic exceeds the degrees of $P$ and $Q$ and the size of some nonconstant coefficient.

In the special case $P = Q$, Theorem \ref{thm: P-P=Q} can be seen as a polynomial version of Schur's theorem over finite fields.
Indeed, the classical theorem of Schur \cite{schur} asserts that the equation $x + y = z$ is partition regular over $\N$.
We have just established partition regularity of the equation $P(x) + P(y) = P(z)$ over finite fields whenever $P$ is FF-intersective.
While the property of being FF-intersective depends on the characteristic $p$, it is automatically satisfied for polynomials with zero constant term.
Hence, Corollary \ref{cor: polynomial Schur} holds.

The equation $P(x) - P(y) = Q(z)$ is often not partition regular (and may not even be solvable) over $\N$.
A key fact leveraged in the proof of Theorem \ref{thm: P-P=Q} is that polynomials take on a positive proportion of values in finite fields, something that is far from the case in $\N$.
It remains an interesting and difficult open problem, asked by Erd\H{o}s and Graham in \cite{eg}, whether the Pythagorean equation $x^2 + y^2 = z^2$ is partition regular over $\N$.
(This was settled with a computer-assisted proof in the case of 2-colorings in \cite{hkm} but is wide open for 3 or more colors.) \\

Some comments are in order on the use of ultraproducts in the proofs of the aforementioned partition regularity results.
The basic strategy we have taken is to discard those colors that have zero Loeb measure in the ultraproduct and then to use recurrence along the polynomial $Q$ to find the desired points $x,y,z$ with $P(x) - P(y) = Q(z)$.
One may be tempted to carry out this strategy in purely finitary terms, avoiding the use of ultraproducts and Loeb measure.
Unfortunately, this does not work (at least in its most straightforward implementation).
The following discussion illuminates the issues that arise.
Fix a FF$_p$-intersective polynomial $Q(x) \in \F_p[x]$.
For simplicity, we will consider $P(x) = x$.
Let $r \in \N$.
Suppose $k \in \N$ is large and an $r$-coloring $\F_{p^k} = \bigcup_{i=1}^r{C_i}$ is given.
We wish to use a function $\Phi : \N \to \left[ 0, \frac{1}{r} \right]$ as a cutoff for distinguishing ``large'' from ``small'' color classes.
That is, we will consider a color class $C_i$ large if $|C_i| \ge \Phi(k) p^k$ and small if $|C_i| < \Phi(k) p^k$.
Without loss of generality, we may assume $C_1, \dots, C_s$ are large and $C_{s+1}, \dots, C_r$ are small for some $s \in \{1, \dots, r\}$.
(The requirement that $\Phi(k) \le \frac{1}{r}$ guarantees that at least one color class is large.)
We now proceed as in the proof of Theorem \ref{thm: P-P=Q}, using ``large'' as a replacement for having positive Loeb measure.
Let $A = C_1 \times \dots \times C_s$.
Proposition \ref{prop: diagonal action} gives the bound
\begin{equation*}
	\E_{z \in \F_{p^k}} \frac{\left| A \cap \left( A + \left( Q(z), \dots, Q(z) \right) \right) \right|}{p^{sk}} \ge \left( \frac{|A|}{p^{sk}} \right)^2 + O \left( p^{-k/2} \right).
\end{equation*}
Since $\left| A \cap \left( A + \left( Q(z), \dots, Q(z) \right) \right) \right| \le |A|$ for each $z \in \F_{p^k}$, we deduce that
\begin{align*}
	\left| \left\{ z \in \F_{p^k} : A \cap \left( A + \left( Q(z), \dots, Q(z) \right) \right) \ne \es \right\} \right|
	 & \ge \frac{p^{(s+1)k}}{|A|} \E_{z \in \F_{p^k}} \frac{\left| A \cap \left( A + \left( Q(z), \dots, Q(z) \right) \right) \right|}{p^{sk}} \\
	 & \ge \frac{|A|}{p^{(s-1)k}} + O \left( \frac{p^{\left( s + \frac{1}{2} \right) k}}{|A|} \right) \\
	 & \ge \Phi(k)^s p^k + O \left( \Phi(k)^{-s} p^{k/2} \right),
\end{align*}
where in the last step we have used the bound $|A| \ge \Phi(k)^s p^{sk}$.
In order to complete the argument, we want to find $z \in \bigcup_{i=1}^s{C_i}$ satisfying $A \cap \left( A + \left( Q(z), \dots, Q(z) \right) \right) \ne \es$.
To that end, one would like to show
\begin{equation*}
	\left| \left\{ z \in \F_{p^k} : A \cap \left( A + \left( Q(z), \dots, Q(z) \right) \right) \ne \es \right\} \right|
	 > \left| \bigcup_{i=s+1}^r{C_i} \right|.
\end{equation*}
The total size of the small color classes is bounded by
\begin{equation*}
	\left| \bigcup_{i=s+1}^r{C_i} \right| \le (r-1) \Phi(k) p^k.
\end{equation*}
The goal, then, is to choose the function $\Phi : \N \to \left[ 0,\frac{1}{r} \right]$ so that
\begin{equation*}
	\Phi(k)^s p^k > (r-1) \Phi(k) p^k + O \left( \Phi(k)^{-s} p^{k/2} \right).
\end{equation*}
Dividing by $p^k$, this reduces to the inequality
\begin{equation*}
	\Phi(k)^s > (r-1) \Phi(k) + O \left( \Phi(k)^{-s} p^{-k/2} \right).
\end{equation*}
But for $r \ge 2$, this requires $\Phi(k) > 1$, which violates the condition that $0 \le \Phi(k) \le \frac{1}{r}$.

Working with the ultraproduct allows us to replace ``small'' with measure zero.
This is crucial, as we have just seen that ``small'' contributions in the finitary setting may accumulate and overtake individual ``large'' terms.
In contrast, finite unions of measure zero sets remain of measure zero.
However, our infinitary methods come at a cost: we are unable to provide any quantitative control on the values $K$ and $c$ appearing in the statement of Theorem \ref{thm: P-P=Q} and related corollaries.
It is therefore an interesting problem to obtain a purely finitary proof of Theorem \ref{thm: P-P=Q} with effective bounds.


\section*{Acknowledgments}

This work was initiated and substantial portions were carried out while the authors were at the Institute for Advanced Study in Princeton, NJ, participating in the special year program, ``Applications of Dynamics in Number Theory and Algebraic Geometry.''
The first author acknowledges support from the National Science Foundation (Grant No. DMS-1926686) and the Swiss National Science Foundation (Grant No. TMSGI2-211214).
We would also like to thank Peter Sarnak for pointing us to the work of Lang and Weil \cite{lw} and for insightful discussions that helped shape Section \ref{sec: partition regularity}.


\end{document}